\documentclass[11pt]{article}
\usepackage[margin=1.1in]{geometry}
\usepackage{amsmath,amssymb,amsthm,mathtools}
\usepackage[colorlinks=true,linkcolor=blue,citecolor=blue,urlcolor=blue]{hyperref}

\usepackage{enumitem}

\newtheorem{theorem}{Theorem}
\newtheorem{proposition}{Proposition}
\newtheorem{lemma}{Lemma}

\theoremstyle{definition}
\newtheorem{remark}{Remark}
\newtheorem{example}{Example}

\newtheorem{assumption}{Assumption}
\newtheorem{problem}{Problem}

\newcommand{\E}{\mathbb E}
\newcommand{\R}{\mathbb R}
\newcommand{\dd}{\mathrm d}
\newcommand{\law}{\mathrm{Law}}
\newcommand{\U}{\mathrm U}

\newcommand{\Psym}{\mathcal P_{\mathrm{sym}}}

\newcommand{\Pcsym}{\mathcal P_{\mathrm{c},\mathrm{sym}}}
\newcommand{\piant}{\pi_{\mathrm{ant}}}
\newcommand{\picom}{\pi_{\mathrm{com}}}

\usepackage{tikz}

\title{Diamond Transports in Quadratic-Form and Distorted Optimal Transport}
\author{Ruodu Wang\thanks{Department of Statistics and Actuarial Science, University of Waterloo, Canada. Email: \texttt{wang@uwaterloo.ca}.}\and 
Zhenyuan Zhang\thanks{Department of Mathematics, Stanford University, USA. Email: \texttt{zzy@stanford.edu}.}
 } 
\date{}

\begin{document}
\maketitle

\begin{abstract}

The diamond transport is generated by the uniform law on a diamond-shaped copula support. Since a classical optimal transport (OT) objective  is affine in the coupling, this transport cannot be the unique minimizer in the classical setting. We study a broader family of transports, called diamond-type transports, in non-classical settings such as quadratic-form optimal transport (QOT) and distorted optimal transport (DOT), which are generally nonconvex. Our main results are within the QOT framework: for symmetric one-dimensional marginals, the diamond transport is an optimizer for a large class of QOT problems whose costs depend on within-coordinate distances. Examples include product costs  under positive-definiteness and convexity conditions and, in particular, mixed rectangular costs. For rectangular costs, we show that the diamond transport is the unique minimizer except for boundary cases. In the DOT framework, diamond-type transports are minimizers for a natural class of cost, and the diamond transport is the unique minimizer in specialized examples. We also identify the intersection between DOT and QOT, which corresponds precisely to quadratic distortion functions.


\smallskip
\noindent \textbf{MSC classification 2020}: 49Q22; 62H05
\end{abstract}

\section{Introduction}
The diamond copula $C_{\mathrm{dia}}$ is the uniform distribution on the boundary of the two-dimensional diamond $$\{(u,v)\in [0,1]^2: |2u-1|+|2v-1|=1\}.$$ 
\sloppy For marginals \(\mu,\nu\) on $\R$ with distributions $F_\mu,F_\nu$, the {diamond transport} is
\(\pi_{\mathrm{dia}}=(F_\mu^{-1},F_\nu^{-1})_\# C_{\mathrm{dia}}\). 
  Figure \ref{fig:diamond-copula} illustrates $\pi_{\rm dia}$ and $C_{\rm dia}$. 

 For probability measures $\mu,\nu$, we let $\Pi(\mu,\nu)$ denote the set of all couplings of $\mu,\nu$, i.e., joint distributions with marginals given by $\mu,\nu$. 
A classical optimal transport (OT) problem \cite{santambrogio2015optimal}, for a cost function \(c\), is
\[
  \text{minimize }  \iint c(x,y) \,\dd\pi(x,y)
  \quad\text{over }\pi\in\Pi(\mu,\nu).
\]
  In general, the diamond transport cannot be the unique minimizer of any classical OT problem, because it is the mixture of two other copulas (see Figure \ref{fig:2}), and the OT problem is affine in $\pi$. For instance, when $\mu,\nu$ are symmetric about zero, the diamond transport is a minimizer of $\int (|x|+|y|)^2\dd \pi(x,y)
 $, but so are the transports associated with the copulas in Figure \ref{fig:2} and their weighted averages; see Theorem \ref{thm:absolute-value-classic} for a more general result.

\begin{figure}[t]
\centering
\setlength{\unitlength}{0.78mm}
\begin{picture}(160,84)
\thinlines
\put(4,16){\framebox(56,56){}}
\put(98,16){\framebox(56,56){}}
\thicklines
\qbezier[60](32,72)(54,69)(60,44)
\qbezier[60](60,44)(55,20)(32,16)
\qbezier[60](32,16)(9,20)(4,44)
\qbezier[60](4,44)(10,69)(32,72)
\put(126,72){\line(1,-1){28}}
\put(154,44){\line(-1,-1){28}}
\put(126,16){\line(-1,1){28}}
\put(98,44){\line(1,1){28}}
\thinlines
\put(68,44){\vector(1,0){20}}
\put(78,52){\makebox(0,0){\((F_\mu,F_\nu)\)}}
\put(32,80){\makebox(0,0){nonuniform physical marginals}}
\put(126,80){\makebox(0,0){copula coordinates}}
\put(32,8){\makebox(0,0){physical space}}
\put(126,8){\makebox(0,0){copula space}}
\end{picture}
\caption{Support of the diamond transport (left) and the diamond copula (right).}
\label{fig:diamond-copula}

\bigskip

\begin{tikzpicture}[scale=4, line cap=round, line join=round]

\begin{scope}
  
  \draw[thin] (0,0) rectangle (1,1);

  \draw[very thick] (0,0.5) -- (0.5,1);
  \draw[very thick] (0.5,0) -- (1,0.5);

  \foreach \x/\lab in {0/$0$,0.5/$\frac12$,1/$1$} {
    \draw[thin] (\x,0) -- (\x,-0.015);
    \node[below] at (\x,-0.015) {\scriptsize \lab};
  }
  \foreach \y/\lab in {0/$0$,0.5/$\frac12$,1/$1$} {
    \draw[thin] (0,\y) -- (-0.015,\y);
    \node[left] at (-0.015,\y) {\scriptsize \lab};
  }

  \node[below] at (0.5,-0.16) {\small $u$};
  \node[left] at (-0.12,0.5) {\small $v$};
  \node[below] at (0.5,-0.30) {\small Support of $C_1$};
\end{scope}

\begin{scope}[xshift=1.45cm]
  
  \draw[thin] (0,0) rectangle (1,1);

  \draw[very thick] (0,0.5) -- (0.5,0);
  \draw[very thick] (0.5,1) -- (1,0.5);

  \foreach \x/\lab in {0/$0$,0.5/$\frac12$,1/$1$} {
    \draw[thin] (\x,0) -- (\x,-0.015);
    \node[below] at (\x,-0.015) {\scriptsize \lab};
  }
  \foreach \y/\lab in {0/$0$,0.5/$\frac12$,1/$1$} {
    \draw[thin] (0,\y) -- (-0.015,\y);
    \node[left] at (-0.015,\y) {\scriptsize \lab};
  }

  \node[below] at (0.5,-0.16) {\small $u$};
  \node[left] at (-0.12,0.5) {\small $v$};
  \node[below] at (0.5,-0.30) {\small Support of $C_2$};
\end{scope}

\end{tikzpicture}

\caption{Two copula supports whose average is the diamond copula:
$C_{\rm dia}= (C_1 +   C_2)/2$.}
\label{fig:2}
\end{figure}

Outside the classical OT framework, as we  will show in this paper, the diamond transport can be the unique minimizer of certain classes of non-classical OT problems. In particular, this happens in the framework of
quadratic-form optimal
transport (QOT), introduced in \cite{wangzhangQOT}, as well as the framework of
 distorted optimal transport (DOT), introduced in \cite{liu2023distorted}.

Our paper analyzes the diamond transport and associated transports in QOT and DOT, with 
 QOT as the main focus. 
 Given \(\mu,\nu\), and a cost function \(c\), the QOT problem is 
\[
  \text{minimize }\quad
  \mathcal Q_c(\pi):=\iint c(x,y,x',y')\,\dd\pi(x,y)\,\dd\pi(x',y')
  \quad\text{over }\pi\in\Pi(\mu,\nu).
\]
This quadratic dependence on the coupling makes QOT structurally different
from, and more difficult than,  classical Kantorovich OT.  It also
contains models closely related to Gromov--Wasserstein distances
\cite{memoli2007use,memoli2011gromov}, quadratic assignment problems (QAP)
\cite{burkard2012assignment,cela1998quadratic}, and quadratically
regularized optimal transport
\cite{blondel2018smooth,essid2018quadratically,lorenz2021quadratically}. On the other hand, QOT problems are in general non-convex and can be NP-hard, so explicit solutions appear rare and attractive.
An important observation is that QOT is no longer affine in
the coupling.  A first-order condition may still be governed by the same
antimonotone matching of absolute values as in classical OT, but the quadratic
term can add strict convexity on the set of couplings and thereby select the
uniform sign structure of the diamond transport.

A natural subclass consists of costs of the form \(c(x,y,x',y')=h(|x-x'|,|y-y'|)\), where the within-coordinate
distances measure paired variations. In this case, the
total cost $\mathcal Q_c(\pi)$ can encode the discrepancy, similarity, or fairness of $\pi$ through pairwise
comparisons. For example, a cost in this class could measure the
difference between \(Y\) and \(Y'\) while the separation of \(X\) and \(X'\)
controls how this difference is discounted \cite[Appendix C.1]{wangzhangQOT}. As another example, the case $h(u,v)=|u^q-v^q|^p$ leads to the Gromov--Wasserstein transport cost on $(\R,|\cdot|)$. Results in  
\cite[Section 6]{wangzhangQOT} further showed that some examples in this class are explicitly solved by the diamond transport. For instance, a maximizer of the Gromov--Wasserstein distance with $p=2,\,q\in(1,2)$ and symmetric marginals is given by the diamond coupling. What makes such results attractive is that the diamond transport is not Monge (i.e., not representable by a map) and is not uniquely selected by classical OT problems.

We extend our study to the class of   \emph{diamond-type transports}, namely transports that are absolutely continuous with respect to the diamond transport, as well as the framework of   DOT.
The DOT problem is formulated by replacing the expectation of the transport cost
with  a distorted expectation. Distorted expectations are
Choquet integrals with respect to capacities obtained by distorting
probabilities, building on the Choquet integral of \cite{choquet1954theory} and the
nonadditive expectation theory, popular in economic decision theory \cite{quiggin1982anticipated,yaari1987dual,schmeidler1986integral,schmeidler1989subjective} and risk measures \cite{acerbi2002spectral,mcneil2015quantitative}.  The DOT
nonlinearity depends only on the distribution of \(c(X,Y)\), whereas QOT
depends on pairwise interactions between two independent copies of the
coupling.  This contrast helps separate three mechanisms: linearity in classic
OT, positive-definiteness and quadratic convexity in QOT, and scalar distortion
in DOT.

We summarize our main technical contributions below.
In most results, we assume the marginals are symmetric about zero. 
The most technically heavy innovations are the QOT results. 
\begin{itemize}
\item We  introduce the class of diamond-type transports, and show in Theorem \ref{thm:absolute-value-classic} that they are the minimizers of a class of OT problems concerning the absolute values of both components. 
    \item We give a general convex criterion for diamond optimality (Theorem \ref{thm:main}).  The argument consists of two steps: convexity reduces global optimality
to a first-variation inequality at the diamond transport, and a quadrant
rearrangement argument proves that inequality. A crucial property we use is that for symmetric one-dimensional marginals, the diamond transport couples their absolute values antimonotonically and then assigns independent signs.  This gives a shorter and more powerful proof of the rectangular examples in \cite[Theorem 12]{wangzhangQOT}. 

\item Following Theorem \ref{thm:main}, we provide a neat and general sufficient condition for the diamond optimality and uniqueness for costs of the form $\psi_1(|x-x'|)\psi_2(|y-y'|)$ (Theorem \ref{thm:taylor-products}). More precisely, we show that as long as $\psi_i,~i=1,2$ satisfy $\psi_i(0)=0$, are continuous, nondecreasing, convex, and $C^2$ on $(0,\infty)$, and are such that for every \(\alpha>0\), $(u,v)\mapsto e^{-\alpha\psi_i(|u-v|)}$
is positive definite on \(\mathbb R\), then the diamond transport is a minimizer. In particular, this includes the mixed rectangular
cost $|x-x'|^p|y-y'|^q,~p,q\in(1,2]$. The proof consists of a three-fold limit argument: first truncate the marginals, then regularize $\psi_i,~i=1,2$ to gain enough convexity, and finally a Taylor expansion trick that reduces this regularized optimization problem to a limiting application of Theorem \ref{thm:main}.

\item We characterize the uniqueness of the solution to the mixed rectangular
cost \(|x-x'|^p|y-y'|^q\) (Theorem \ref{thm:q-rectangular}).  We prove that the mixed rectangular
cost has the diamond transport as the unique minimizer
when \(1<p,q<2\), under the natural \(p\)- and \(q\)-moment assumptions.
This resolves a conjecture in \cite[Appendix D(iv)(b)]{wangzhangQOT}.  We also give a complete uniqueness criterion for the
boundary cases in which one or both exponents equal \(2\).  The proof relies on 
the fact that $|x-x'|^p,~p\in(1,2)$ is a strictly negative type kernel.

\item 
Theorem \ref{thm:dot-qot-overlap} shows that DOT and QOT overlap precisely when the distortion function is a quadratic function.   In Theorem  \ref{thm:dot-diamond-type}, we show that diamond-type transports minimize a class of DOT problems formulated on the absolute value of both components.  
We also give a specialized example of DOT (Example \ref{thm:dot-unique-diamond}) in which the diamond transport is the unique minimizer.
\end{itemize}

Although the previous work \cite{wangzhangQOT} provided explicitly solvable examples of diamond optimality, the QOT results in this paper appear significantly stronger for the following reasons. First, Theorem \ref{thm:main} is more powerful than \cite[Theorem 12]{wangzhangQOT} (Remark \ref{rem:comp1}) and admits a much simpler proof. Second, along the lines of Theorem \ref{thm:taylor-products} and its subsequent examples, \cite[Theorem 14]{wangzhangQOT} only managed to prove rectangular costs with $p=q$ under suboptimal moment assumptions, and hence Theorem \ref{thm:taylor-products} is much more general. Third, for the main QOT results, we provide sufficient conditions for the uniqueness of the solution, which was not addressed in \cite{wangzhangQOT}. Fourth, we provide an abundant number of explicit examples, considerably enlarging the class in \cite{wangzhangQOT}.

Diamond transports have not appeared in the DOT framework before, and our paper offers the first attempt to analyze them. In terms of diamond transport in DOT, a full picture is far   from   clear. Our results and examples on DOT should be seen as exploratory.
Nevertheless, they illustrate that diamond transports appear not only in QOT, but also in other nonlinear formulations of OT.

The remainder of this paper is organized as follows.  Section~\ref{sec:setup} fixes the
notation and basic definitions.
Section~\ref{sec:diamond-type} formally defines diamond-type transports, and shows  the equivalence between an absolute-value principle
and the diamond-type.
Section~\ref{sec:main-results} states the
main QOT results and examples, including comparisons with
\cite{wangzhangQOT}.  
Section~\ref{sec:dot} develops the DOT counterpart,
including convexity, DOT--QOT overlap, and diamond-type optimality.
Section~\ref{sec:QAP} briefly discusses a discrete QAP problem inspired by the diamond transport.  
Section~\ref{sec:proofs} gives the proofs of the QOT
results in Section~\ref{sec:main-results}, which are our most technical results.
Section~\ref{sec:conclusion} concludes the paper.

\section{Setup and notation}\label{sec:setup}

This section fixes the notation and definitions used throughout the paper.  A probability
measure on \(\R\) is {symmetric about zero}, or simply {symmetric}, if it is invariant under the
map \(x\mapsto -x\).  Let $\Psym(\R)$ be the set of symmetric probability measures on $\R$, and \(\Pcsym(\R)\) be the set of symmetric compactly supported probability measures on $\R$.  For probability measures 
\(\mu,\nu\) on $\R$, let \(\Pi(\mu,\nu)\) be the set of
\emph{couplings} of \(\mu\) and \(\nu\).  For \(q>0\), let
\(\mathcal P_q(\R)\) be the set of probability measures on \(\R\) with
finite \(q\)-th absolute moment.  If \(\lambda\) is a compactly supported probability measure on $\R$, its
support diameter is
\(\operatorname{diam}(\operatorname{supp}\lambda)
:=\sup_{x,x'\in\operatorname{supp}\lambda}|x-x'|\).
Thus, when we say that the support diameters of \(\mu\) and \(\nu\) are denoted
by \(D_x\) and \(D_y\), this means
\(D_x=\operatorname{diam}(\operatorname{supp}\mu)\) and
\(D_y=\operatorname{diam}(\operatorname{supp}\nu)\). 

For a random element
\(Z\), write \(\law(Z)\) for its probability law.  For a probability law
\(\lambda\), write $F_{\lambda}$ for its cumulative distribution function and \(F_\lambda^{-1}\) for its usual left-continuous quantile
function on \((0,1)\).  For probability laws
\(\mu,\nu\) on \(\R\),
define the comonotone and antimonotone couplings as $\picom=\law(F_\mu^{-1}(U),F_\nu^{-1}(U))$ and $\piant=\piant(\mu,\nu)=\law(F_\mu^{-1}(U),F_\nu^{-1}(1-U))$, respectively, 
where \(U\sim\U(0,1)\). For a coupling \(\gamma\) on \([0,\infty)^2\), write
\(\gamma^\top\) for its transpose, namely the pushforward of \(\gamma\)
under the coordinate-swap map \((r,t)\mapsto(t,r)\). A coupling $\pi$ of \(\mu\) and \(\nu\) is
\emph{Monge from \(\mu\) to \(\nu\)} if it has the form
\((\mathrm{Id},T)_\#\mu\) for a Borel map \(T\).  For
laws \(\mu,\nu\) on \([0,\infty)\), we say that a coupling \(\pi\) of
\(\mu\) and \(\nu\) is \emph{Monge from \(\mu\) to \(\nu\) away
from zero} if there is a Borel \(T:(0,\infty)\to[0,\infty)\) such that
\(\pi\) is supported by
\(\{(r,t): r=0\text{ or }t=T(r)\}\).   For a measurable kernel
\(c:\R^2\times\R^2\to\R\), \(c(z,z')=c(z',z)\), the
\emph{quadratic-form cost (or energy)} is
\begin{equation}
\label{eq:quadratic-form-energy}
  \mathcal Q_c(\pi):=
  \iint c\big((x,y),(x',y')\big)
  \,\dd\pi(x,y)\,\dd\pi(x',y'),
  \qquad \pi\in\Pi(\mu,\nu),
\end{equation}
whenever the integral is well-defined, possibly with value \(+\infty\).
The minimization problems below are all instances of
\eqref{eq:quadratic-form-energy}.

Following \cite{wangzhangQOT}, a cost is called \emph{type-XX} if it has the
form
\begin{equation}
\label{eq:typexx-form}
  c(x,y,x',y')=H(f(x,x'),g(y,y'))
\end{equation}
for some real-valued bivariate functions \(f,g\) and a real-valued function
\(H\).\footnote{ The terminology contrasts with type-XY costs of the form $c(x,y,x',y')=H(f(x,y),g(x',y'))$; those costs are
not studied in this paper.}  The terminology indicates that the cost aggregates information from
the within-\(x\) pair \((x,x')\) and the within-\(y\) pair \((y,y')\).  A
standard example is the Gromov--Wasserstein cost
\cite{memoli2011gromov}.  All costs treated below will have the form
\eqref{eq:typexx-form}.

A kernel \(K\) on a set \(S\) is positive definite if
\(\sum_{i,j=1}^n a_i a_j K(z_i,z_j)\ge0\) for every \(n\), every
\(z_1,\dots,z_n\in S\), and every
\(a_1,\dots,a_n\in\R\).  It is strictly positive definite on finite
signed measures if, for every nonzero finite signed Borel measure
\(\sigma\) supported on \(S\) for which the integral is well-defined, one has
\(\iint K(z,z')\,\dd\sigma(z)\,\dd\sigma(z')>0\).
A function \(F:[0,\infty)\to\R\) is completely monotone if it is
\(C^\infty\) on \((0,\infty)\) and
\((-1)^kF^{(k)}(u)\ge0\) for \(u>0\) and \(k=0,1,2,\ldots\).

For
general symmetric marginals, we use the equivalent radius/sign definition:
if \(\mu_{\mathrm{abs}}\) and \(\nu_{\mathrm{abs}}\) are the laws of
\(|X|\), \(X\sim\mu\), and \(|Y|\), \(Y\sim\nu\), then
\(\pi_{\mathrm{dia}}=\law(\varepsilon A,\eta B)\), where
\(A=F_{\mu_{\mathrm{abs}}}^{-1}(U)\),
\(B=F_{\nu_{\mathrm{abs}}}^{-1}(1-U)\), \(U\sim\U(0,1)\), and
\(\varepsilon,\eta\) are independent Rademacher signs, independent of \(U\). Thus the absolute values are coupled antimonotonically, while the signs are
independent.  In particular, usually neither the diamond transport nor its transpose is Monge; the Monge conditions used later
concern only the absolute-value coupling.

A function \(G\) on a rectangle is \emph{supermodular} if
\[
  G(r_2,s_2)+G(r_1,s_1)
  \ge
  G(r_2,s_1)+G(r_1,s_2)
\]
whenever \(r_1\le r_2\) and \(s_1\le s_2\).  For \(C^2\) functions this is
equivalent to \(\partial^2G/\partial r\,\partial s\ge0\).
A function is \emph{submodular} if the reverse inequality holds.  For a
function defined only on a subset of a rectangle, either term means that the
corresponding four-point inequality holds whenever the four points appearing
in it all belong to the subset.

\section{Diamond-type transports in classical OT}\label{sec:diamond-type}

We first recall the definition of the diamond transport in the Introduction. 
 The \emph{diamond copula} \(C_{\mathrm{dia}}\) is the uniform probability law on
the boundary of the \(\ell^1\)-ball
\(\{(u,v)\in[0,1]^2: |u-1/2|+|v-1/2|=1/2\}\).  
For symmetric atomless marginals \(\mu,\nu\), the \emph{diamond transport} is
\(\pi_{\mathrm{dia}}=(F_\mu^{-1},F_\nu^{-1})_\# C_{\mathrm{dia}}\). In other words,
the copula is \(C_{\mathrm{dia}}\), although the support in physical
\((x,y)\)-coordinates may be curved by the marginal quantiles (see
Figure~\ref{fig:diamond-copula}).

The diamond transport can be enlarged to a natural
class of transports. 
We say that \(\pi\in\Pi(\mu,\nu)\) is \emph{diamond-type} if
\(\pi\ll\pi_{\rm dia}\). 
The next result shows  that for symmetric marginals, this is equivalent to
antimonotone matching of absolute values.  Further, when the marginals are
continuous, it is also equivalent to saying that the copula of \(\pi\) is
supported on \[
  \Delta_{\mathrm{dia}}
  :=
  \{(u,v)\in[0,1]^2: |u-1/2|+|v-1/2|=1/2\},
\] and the diamond transport corresponds
to the uniform law on \(\Delta_{\mathrm{dia}}\).

\begin{theorem}
\label{thm:absolute-value-classic}
Let \(\mu,\nu\in\mathcal{P}_{\rm sym}(\R)\), and set
\(\mu_{\mathrm{abs}}=|\cdot|_\#\mu\) and \(\nu_{\mathrm{abs}}=|\cdot|_\#\nu\). 
Let \(G:[0,\infty)^2\to[0,\infty]\) be lower semi-continuous and strictly
supermodular on \(\operatorname{supp}\mu_{\mathrm{abs}}\times\operatorname{supp}\nu_{\mathrm{abs}}\). Let
\(\gamma_{\rm ant}\in\Pi(\mu_{\mathrm{abs}},\nu_{\mathrm{abs}})\) be the antimonotone coupling and assume
\(\int G(r,s)\,\mathrm{d}\gamma_{\rm ant}(r,s)<\infty\).  For
\(\pi\in\Pi(\mu,\nu)\), the following statements are equivalent:
\begin{enumerate}
\item \(\pi\) minimizes
 $
  \int G(|x|,|y|)\,\mathrm{d}\pi(x,y)
 $
over \(\Pi(\mu,\nu)\);
\item \((|x|,|y|)_\#\pi=\gamma_{\rm ant}\);
\item \(\pi\ll\pi_{\rm dia}\) (i.e., $\pi$ is diamond-type).
\end{enumerate}
If the marginal distributions $\mu,\nu$ are continuous, these conditions are further
equivalent to \begin{enumerate}[resume] \item  the copula of \(\pi\) is supported on
\(\Delta_{\mathrm{dia}}\).\end{enumerate}
\end{theorem}

In particular, the diamond transport is a minimizer to the problem in item (i) of Theorem \ref{thm:absolute-value-classic}.
The equivalence between (i) and (ii) does not require symmetry of
\(\mu,\nu\); symmetry is used only to identify the antimonotone
absolute-value condition  in (ii) with  the  diamond-type condition in (iii).

\begin{proof}
For \(\pi\in\Pi(\mu,\nu)\), let
 $
  \gamma_\pi=(|x|,|y|)_\#\pi .
$
Then \(\gamma_\pi\in\Pi(\mu_{\mathrm{abs}},\nu_{\mathrm{abs}})\) and
\[
  \int G(|x|,|y|)\,\mathrm{d}\pi(x,y)
  =
  \int G(r,s)\,\mathrm{d}\gamma_\pi(r,s).
\]
Conversely, every \(\gamma\in\Pi(\mu_{\mathrm{abs}},\nu_{\mathrm{abs}})\) can be lifted to a coupling of
\(\mu\) and \(\nu\).  Indeed, disintegrate \(\mu\) and \(\nu\) with respect to
the maps \(x\mapsto |x|\) and \(y\mapsto |y|\), and write
\(\kappa_\mu(r,\mathrm{d}x)\) and \(\kappa_\nu(s,\mathrm{d}y)\) for regular
conditional laws supported on \(\{|x|=r\}\) and \(\{|y|=s\}\).  Then
\[
  \int \kappa_\mu(r,\mathrm{d}x)\kappa_\nu(s,\mathrm{d}y)
  \,\mathrm{d}\gamma(r,s)
\]
belongs to \(\Pi(\mu,\nu)\) and has absolute-value coupling \(\gamma\).  Hence
minimizing over \(\Pi(\mu,\nu)\) is equivalent to minimizing
\(\int G(r,s)\,\mathrm{d}\gamma(r,s)\) over \(\Pi(\mu_{\mathrm{abs}},\nu_{\mathrm{abs}})\).  The
one-dimensional rearrangement inequality for lower semi-continuous strictly
supermodular costs (\cite[Theorem~3.1]{puccetti2015extremal}) gives the unique minimizer \(\gamma_{\rm ant}\).  This
proves the equivalence of (i) and (ii).

We next show that (ii) and (iii) are equivalent.  Since \(\mu\) and \(\nu\)
are symmetric, the diamond transport has the disintegration
\[
  \pi_{\rm dia}(\mathrm{d}x,\mathrm{d}y)
  =
  \int K_{r,s}(\mathrm{d}x,\mathrm{d}y)
  \,\mathrm{d}\gamma_{\rm ant}(r,s),
\]
where \(K_{r,s}\) is the uniform law on the distinct points
\(\{(x,y): |x|=r,\ |y|=s\}\).  If (ii) holds, disintegrate \(\pi\)
conditionally on \((|X|,|Y|)=(r,s)\).  Each conditional law is supported on
\(\{(x,y): |x|=r,\ |y|=s\}\), and is therefore absolutely continuous with
respect to \(K_{r,s}\).  Since the radius law is \(\gamma_{\rm ant}\), this
gives \(\pi\ll\pi_{\rm dia}\).

Conversely, suppose \(\pi\ll\pi_{\rm dia}\).  Then
\(\gamma_\pi\ll\gamma_{\rm ant}\).  The antimonotone coupling
\(\gamma_{\rm ant}\) is supported on an antimonotone set: it is the law of
\((F_{\mu_{\mathrm{abs}}}^{-1}(U),F_{\nu_{\mathrm{abs}}}^{-1}(1-U))\), \(U\sim\U(0,1)\).  Hence
\(\gamma_\pi\) is also supported on an antimonotone set. Therefore, it is the antimonotone coupling between $\mu_{\mathrm{abs}}$ and $\nu_{\mathrm{abs}}$, and thus equal to
\(\gamma_{\rm ant}\).  This proves (ii) and (iii) are equivalent.
 
It remains  to show that the equivalence between (iii) and (iv)   records the copula formulation when the marginals are
continuous.  Let \((X,Y)\sim\pi\), \(U=F_\mu(X)\), and \(V=F_\nu(Y)\).  Then
\((U,V)\) is the unique copula of \(\pi\).  By symmetry,
\[
  F_{\mu_{\mathrm{abs}}}(|X|)=|2U-1|,
  \qquad
  F_{\nu_{\mathrm{abs}}}(|Y|)=|2V-1|
\]
almost surely.  Because \(\mu_{\mathrm{abs}}\) and \(\nu_{\mathrm{abs}}\) are continuous, the condition
\((|X|,|Y|)_\#\pi=\gamma_{\rm ant}\) is equivalent to
\[
  F_{\mu_{\mathrm{abs}}}(|X|)+F_{\nu_{\mathrm{abs}}}(|Y|)=1
  \quad\text{a.s.}
\]
This is exactly
\[
  |U-1/2|+|V-1/2|=1/2
  \quad\text{a.s.},
\]
which means that the copula of \(\pi\) is supported on
\(\Delta_{\mathrm{dia}}\).  This proves the equivalence between (iii) and (iv).
\end{proof}

\begin{example}
The copula formulation in Theorem~\ref{thm:absolute-value-classic}, item (iv) cannot be
extended to arbitrary atoms.  Let
\[
  \mu=\frac12 \delta_{-1}+\frac12 \delta_1,
  \qquad
  \nu=\frac14(\delta_{-2}+\delta_{-1}+\delta_1+\delta_2).
\]
Then
\[
  \pi_{\rm dia}
  =
\frac18
  \sum_{x\in\{-1,1\}}
  \sum_{y\in\{-2,-1,1,2\}}
  \delta_{(x,y)}.
\]
Define
\[
  \pi
  =
  \frac14\delta_{(1,-1)}
  +\frac14\delta_{(1,1)}
  +\frac14 \delta_{(-1,-2)}
  +\frac14\delta_{(-1,2)}.
\]
Then \(\pi\in\Pi(\mu,\nu)\), \(\pi\ll\pi_{\rm dia}\), and
\((|x|,|y|)_\#\pi=(|x|,|y|)_\#\pi_{\rm dia}\).  Hence \(\pi\) is
diamond-type. 
We will check that no copula \(C\) satisfying
\((F_\mu^{-1},F_\nu^{-1})_\#C=\pi\) can be supported on
\(\Delta_{\mathrm{dia}}\).  Indeed, the quantile cells give
\[
  X=1 \iff U\in(1/2,1],
  \qquad
  |Y|=1 \iff V\in(1/4,3/4].
\]
Thus such a copula would need
\[
  C((1/2,1]\times(1/4,3/4])
  =
  \pi(X=1,\ |Y|=1)
  =
  1/2.
\]
On \(\Delta_{\mathrm{dia}}\), however, the set
\((1/2,1]\times(1/4,3/4]\) is contained in
\([3/4,1]\times[0,1]\).  Since the first marginal of a copula is uniform,
\[
  C((1/2,1]\times(1/4,3/4])
  \le
  C([3/4,1]\times[0,1])
  =
  1/4,
\]
a contradiction.
\end{example}

\section{Diamond transport in QOT}
\label{sec:main-results}

In this section, we state our main QOT results:
\begin{itemize}
    \item Theorem \ref{thm:main} gives an 
abstract convex diamond criterion for costs of the form \(c(x,y,x',y')=h(|x-x'|,|y-y'|)\). 
\item Theorem \ref{thm:taylor-products} turns it into concrete
families of product costs of the form $c(x,y,x',y')
  =
  \psi_1(|x-x'|)\psi_2(|y-y'|)$, employing a Taylor expansion trick. 
  \item  Theorem \ref{thm:q-rectangular} treats
costs of the form $|x-x'|^p|y-y'|^q,~p,q\in(1,2]$ and provides a full characterization of the uniqueness of the solution.
\end{itemize}
We also provide examples following these results.
The proofs are very technical and presented in Section~\ref{sec:proofs}.

\subsection{An abstract criterion for the optimality of the diamond transport}\label{31}

We first state the main assumption on the marginals and the type-XX cost function for the QOT problem. 

\begin{assumption}
\label{ass:convex-diamond}
Let \(\mu,\nu\in\Pcsym(\R)\), with finite support diameters \(D_x,D_y\), and
let 
$$c(x,y,x',y')=h(|x-x'|,|y-y'|).$$
The following conditions hold:
\begin{enumerate}
\item \(h\) is continuous on \([0,D_x]\times[0,D_y]\) and \(C^2\) on
\((0,D_x)\times(0,D_y)\);
\item the kernel \(((x,y),(x',y'))\mapsto c(x,y,x',y')\) is positive
definite on
\(\operatorname{conv}(\operatorname{supp}\mu)\times
\operatorname{conv}(\operatorname{supp}\nu)\);
\item writing \(W(r,s)=\partial^2h(r,s)/\partial r\,\partial s\), \(W\) has
a finite continuous extension to \([0,D_x]\times[0,D_y]\), and this extension
is nonnegative, nondecreasing in each coordinate, and supermodular.
\end{enumerate}
\end{assumption}
Assumption~\ref{ass:convex-diamond} has two roles.  Positive definiteness ensures that $\pi\mapsto \mathcal Q_c(\pi)$ is convex.  The conditions on \(W\) are related to the optimality of the antimonotone coupling of the
radii. Our next result provides sufficient conditions for the diamond transport to be the unique minimizer under Assumption~\ref{ass:convex-diamond}.
\begin{theorem}
\label{thm:main}
 Suppose that
Assumption~\ref{ass:convex-diamond} holds for \(\mu,\nu,c\).  Then the diamond
transport \(\pi_{\mathrm{dia}}\) minimizes \(\mathcal Q_c\) over
\(\Pi(\mu,\nu)\).  If the kernel is strictly positive definite on finite
signed measures supported on \(\operatorname{supp}\mu\times\operatorname{supp}\nu\),
then \(\pi_{\mathrm{dia}}\) is the unique minimizer.
\end{theorem}

\begin{remark}
\label{rem:comp1}
In the compact case, Theorem~\ref{thm:main} recovers the compact specialization
of the convex-diamond class from \cite[Theorem 12]{wangzhangQOT}.  That
result treats
\(c(x,y,x',y')=\phi(|x-x'|^2)\phi(|y-y'|^2)\), where \(\phi\) is completely
monotone and satisfies
\(\phi'(u)+2u\phi''(u)\le0\) on the relevant squared-distance range.  Indeed,
putting \(a(r)=\phi(r^2)\), Schoenberg's theorem and the Schur product theorem
give positive definiteness of the product kernel, while
\(a'(r)=2r\phi'(r^2)\le0\) and
\(a''(r)=2(\phi'(r^2)+2r^2\phi''(r^2))\le0\).  Hence
\(\partial^2(a(r)a(s))/\partial r\,\partial s=a'(r)a'(s)\) satisfies the
monotonicity and supermodularity of Assumption \ref{ass:convex-diamond}(iii).
\end{remark}

\begin{remark}
\label{rem:weighted-square}
Theorem~\ref{thm:main} also gives a simple sufficient criterion for
nonseparable kernels of the form
\[
  c(x,y,x',y')
  =
  F\big(a|x-x'|^2+b|y-y'|^2\big),
  \qquad a,b>0.
\]
In this case, a sufficient condition for the positive definiteness of $c$ is the complete monotonicity of $F$. Indeed, complete monotonicity enters through the Laplace representation
\[
  F(u)=\int_{[0,\infty)} e^{-\lambda u}\,\rho(\mathrm{d}\lambda),
\]
so \(F(a|x-x'|^2+b|y-y'|^2)\) is a positive mixture of Gaussian positive
definite kernels. The remaining assumptions from Assumption \ref{ass:convex-diamond} are satisfied if moreover
\[
  F''(u)+2uF'''(u)\ge0,\qquad 0\le u\le aD_x^2+bD_y^2.
\]
This can be proved via a direct computation, which we omit here for simplicity.

\end{remark}

The next example gives two concrete kernels obtained from this criterion.

\begin{example}
\label{ex:weighted-kernels}
Let \(\mu,\nu\in\Pcsym(\R)\), with finite support diameters
\(D_x,D_y\), and let \(a,b>0\).

\begin{enumerate}
\item Let \(\gamma>0\), \(\beta>0\), and
\[
  c(x,y,x',y')
  =
  \big(\beta+a|x-x'|^2+b|y-y'|^2\big)^{-\gamma}.
\]
If $\beta\ge (2\gamma+3)\max\{aD_x^2,bD_y^2\},$  the diamond transport is the unique QOT minimizer.

\item Set
\[
  c(x,y,x',y')=\exp\big(-(a|x-x'|^2+b|y-y'|^2)\big).
\]
If $\max\{2aD_x^2,2bD_y^2\}\leq 1$, the diamond transport is the unique QOT minimizer.
\end{enumerate}
Note also that the constants \(a,b\) can be absorbed into a
coordinate rescaling of the marginals, but keeping them visible makes the diameter conditions
explicit.    Thus, the Gaussian example (ii) overlaps with
\cite[Example 14(i)]{wangzhangQOT} after choosing weights,\footnote{We remark that the published version of \cite[Example 14(i)]{wangzhangQOT} contains a typo: the condition $-\alpha+2u\alpha^2\leq 0$ should be imposed.} whereas the inverse
power kernel (i) is generally nonseparable and is not covered by that example. 

\end{example}

\subsection{Product-form type-XX costs minimized by the diamond transport}\label{32}

The point of the next result is that Theorem~\ref{thm:main} can be used in
less immediate ways, which can generate larger families of costs with the diamond transport being the optimizer.

In the following, we let $\Xi$ denote the class of functions $\psi: [0,\infty)\to[0,\infty)$ that satisfy the following conditions:
\begin{itemize}
    \item $\psi(0)=0$ and $\psi$ is continuous, nondecreasing, convex, and $C^2$ on $(0,\infty)$;
    \item for every \(\alpha>0\), $(u,v)\mapsto e^{-\alpha\psi(|u-v|)}$
is positive definite on \(\mathbb R\).
\end{itemize}
We say that $\psi$ has \textit{full-support cosine representation} if 
\begin{equation}
\label{eq:cosine-representation}
  \psi(|u|)
  =
  \int_{\R}(1-\cos(tu))\,\Lambda(\dd t),
\end{equation}
where \(\Lambda\) is a positive symmetric L\'{e}vy measure on
\(\R\setminus\{0\}\), Radon on \(\R\setminus\{0\}\), whose support has
closure \(\R\), satisfies
\(\int_{\R}(1\wedge t^2)\,\Lambda(\dd t)<\infty\), and satisfies
\(\int_{\R}(1-\cos(tu))\,\Lambda(\dd t)<\infty\) for every \(u\in\R\). In particular, if $\psi$ has full-support cosine representation, then \(u\mapsto\psi(|u|)\) is continuous negative definite, and $(u,v)\mapsto e^{-\alpha\psi(|u-v|)}$
is positive definite for every \(\alpha>0\).

\begin{example}\label{ex:pre}
    Let $\psi(u)=u^r,\,r\in[1,2]$. Then clearly $\psi(0)=0$ and $\psi$ is continuous, nondecreasing, convex, and $C^2$ on $(0,\infty)$. The representation
\[
  |u|^r=c_r\int_{\R}(1-\cos(tu))|t|^{-1-r}\,\dd t,
  \qquad \text{for some }c_r>0,~r\in(0,2)
\] gives continuous negative definiteness and a
full-support cosine representation; see
\cite[Chapter~4]{schilling2012bernstein}. The case $r=2$ belongs to $\Xi$ (by the positive definiteness of the Gaussian kernel) but does not have a full-support cosine representation. Other examples that satisfy both conditions include:
\begin{itemize}
    \item $\psi(r)=(1+\lambda r^{a})^{\theta}-1$, where $a\in(1,2],\,\theta\in(0,1),\,\lambda>0$, and $a\theta\in[1,2)$;
    \item $\psi(r)=1-e^{-\lambda r^p}$, where $p\in(1,2)$ and $\lambda>0$; note that $\psi$ is convex only on $[0,(\frac{p-1}{\lambda p})^{1/p}]$;
    \item $\psi(r)=\log(1+\lambda r^p)$, where $p\in(1,2)$ and $\lambda>0$; note that $\psi$ is convex only on $[0,(\frac{p-1}{\lambda})^{1/p}]$.
\end{itemize}
See Example \ref{ex:product-profiles} below.
\end{example}

\begin{theorem}\label{thm:taylor-products}
Let \(\mu,\nu\in\Psym(\R)\) and $\psi_1,\psi_2\in\Xi$.   Then the following statements hold.
\begin{enumerate}
    \item The diamond transport
\(\pi_{\mathrm{dia}}\) minimizes, in the extended sense,
\begin{align}
     \mathcal Q(\pi)
  =
  \iint
  \psi_1(|x-x'|)\psi_2(|y-y'|)
  \,\dd\pi(x,y)\,\dd\pi(x',y')\label{eq:taylor-product-cost}
\end{align}
over \(\Pi(\mu,\nu)\).
\item If, moreover,
\begin{align}
    \int \psi_1(|x|)\,\mu(\dd x)<\infty
  \qquad\text{and}\qquad
  \int \psi_2(|y|)\,\nu(\dd y)<\infty,\label{eq:a2}
\end{align}
then \(\mathcal Q(\pi_{\mathrm{dia}})<\infty\).
\item If, in addition to \eqref{eq:a2}, each
\(u\mapsto\psi_i(|u|)\) has the full-support cosine representation, then \(\pi_{\mathrm{dia}}\) is the unique
minimizer.
\end{enumerate}
  
\end{theorem}

The relation with Theorem~\ref{thm:main} is indirect but simple.  For small
\(\alpha>0\), the exponential kernel
\(\exp(-\alpha\psi_1(|x-x'|)-\alpha\psi_2(|y-y'|))\) satisfies
the conditions in Assumption \ref{ass:convex-diamond}.  If
\[
  A:=\psi_1(|x-x'|)\qquad\text{and}\qquad B:=\psi_2(|y-y'|),
\]
then \(\exp(-\alpha(A+B))=1-\alpha(A+B)+\alpha^2(A+B)^2/2+O(\alpha^3)\)
uniformly on the compact support.  In the quadratic-form energy, the terms involving
only \(A\) or only \(B\) depend only on the fixed marginals but not on the coupling $\pi$.  After
cancelling those marginal terms, the order-\(\alpha^2\) part is exactly
\(\iint AB\,\dd\pi\,\dd\pi\), which is the product cost
\eqref{eq:taylor-product-cost}.

This is the same second-order mechanism used in \cite[Theorem 14]{wangzhangQOT} for the
same-exponent \(q\)-rectangular cost.  The difference is that the present
statement keeps the two one-dimensional profiles separate.

The next example records concrete profile choices covered by
Theorem~\ref{thm:taylor-products}.

\begin{example}
\label{ex:product-profiles}
In each of the following cases the type-XX cost
\(c(x,y,x',y')=\psi_1(|x-x'|)\psi_2(|y-y'|)\) has the unique minimizer
\(\pi_{\mathrm{dia}}\).
\begin{enumerate}
\item Let \(p,q\in(1,2)\), \(\mu\in\mathcal P_p(\R)\cap\Psym(\R)\), and
\(\nu\in\mathcal P_q(\R)\cap\Psym(\R)\).  Then \(\psi_1(r)=r^p\) and \(\psi_2(r)=r^q\) give
\[
  c(x,y,x',y')=|x-x'|^p|y-y'|^q .
\]

\item  Assume that
\(1<a_i\le2\), \(0<\theta_i<1\), \(\lambda_i>0\), and
\(1\leq a_i\theta_i<2\), \(i=1,2\). Let \(\mu\in\mathcal P_{a_1\theta_1}(\R)\cap\Psym(\R)\) and
\(\nu\in\mathcal P_{a_2\theta_2}(\R)\cap\Psym(\R)\).  Put
\[
  \psi_i(r)=(1+\lambda_i r^{a_i})^{\theta_i}-1,\qquad r\ge0.
\]
Then the corresponding product cost has the unique minimizer
\(\pi_{\mathrm{dia}}\). 

\item Let \(\mu,\nu\in\Pcsym(\R)\), with support diameters \(D_x,D_y\).  If
\(p,q\in(1,2)\), \(\lambda,\eta>0\),
\(\psi_1(r)=1-e^{-\lambda r^p}\), and
\(\psi_2(r)=1-e^{-\eta r^q}\), then the same conclusion holds for
\[
  c(x,y,x',y')
  =
  (1-e^{-\lambda |x-x'|^p})(1-e^{-\eta |y-y'|^q})
\]
provided \(\lambda pD_x^p<p-1\) and \(\eta qD_y^q<q-1\). Up to terms
depending only on the marginals, this is equivalent to the product
exponential cost
\(e^{-\lambda |x-x'|^p-\eta |y-y'|^q}\); the analogous boundary \(p=q=2\)
is the Gaussian product case covered directly by Theorem~\ref{thm:main}; see Example \ref{ex:weighted-kernels}(ii).

\item Let \(\mu,\nu\in\Pcsym(\R)\), with support diameters \(D_x,D_y\).  If
\(p,q\in(1,2)\), \(\lambda,\eta>0\),
\(\psi_1(r)=\log(1+\lambda r^p)\), and
\(\psi_2(r)=\log(1+\eta r^q)\), then the same conclusion holds provided
\(\lambda D_x^p<p-1\) and \(\eta D_y^q<q-1\).
\end{enumerate}
\end{example}

\subsection{Rectangular power-type costs and uniqueness of the solution}\label{33}

In our next result, we focus on a specific class of cost functions $c(x,y,x',y'):=|x-x'|^p|y-y'|^q ,~p,q\in(1,2].$ We show the optimality of the diamond transport and study the uniqueness when the marginals are symmetric and satisfy the natural moment constraints. This result improves upon  \cite[Theorem 14]{wangzhangQOT} in three ways: weaker moment assumptions, allowing for mixed exponents $p\neq q$, and a full characterization of the uniqueness. 

\begin{theorem}
\label{thm:q-rectangular}
Let \(1<p,q\le2\), 
\(\mu\in\mathcal P_p(\R)\cap \Psym(\R)\), and \(\nu\in\mathcal P_q(\R)\cap \Psym(\R)\).  Let
\[
  c_{p,q}(x,y,x',y'):=|x-x'|^p|y-y'|^q .
\]
Then the diamond transport \(\pi_{\mathrm{dia}}\) minimizes
\(\mathcal Q_{c_{p,q}}\) over \(\Pi(\mu,\nu)\), and
\(\mathcal Q_{c_{p,q}}(\pi_{\mathrm{dia}})<\infty\).
Let \(X\sim\mu\), \(Y\sim\nu\), and let
\(\gamma^-:=\piant(\law(|X|),\law(|Y|))\). The following statements hold.
\begin{enumerate}
\item If \(1<p,q<2\), then \(\pi_{\mathrm{dia}}\) is the unique minimizer.
\item If \(1<p<2=q\), then a coupling \(\pi=\law(X,Y)\) is a minimizer if and only if
\(\law(|X|,|Y|)=\gamma^-\) and \(\E_\pi[Y\mid X]=0\) \(\pi\)-a.s.
Consequently, uniqueness holds if and only if \(\gamma^-\) is Monge
from \(\law(|X|)\) to \(\law(|Y|)\) away from zero.

\item If \(p=q=2\), then a coupling \(\pi=\law(X,Y)\) is a minimizer if and only if
\(\law(|X|,|Y|)=\gamma^-\) and \(\E_\pi[XY]=0\).  Uniqueness holds if and only
if \(\gamma^-\) is Monge from \(\law(|X|)\) to \(\law(|Y|)\) away from
zero, the transpose of \(\gamma^-\) is Monge from \(\law(|Y|)\) to
\(\law(|X|)\) away from zero, and
\(\gamma^-|_{(0,\infty)^2}\) is either zero or concentrated at a single
point.
\end{enumerate}
\end{theorem}

The case \(1<q<2=p\) is symmetric to part (ii) by switching the roles of $\mu$ and $\nu$.

The minimization part extends Example~\ref{ex:product-profiles}(i) from
compact marginals and the case \(p,q<2\).  The proof first handles compact
marginals and then passes to general marginals by truncation.  For
\(p,q<2\), uniqueness is proved by comparing two minimizers.  Their
difference is a signed measure with zero \(x\)- and \(y\)-marginals.  The
one-dimensional kernels \(|x-x'|^p\) and \(|y-y'|^q\) are strictly negative
type kernels; after the zero-marginal cancellations, their product gives a
nonnegative quadratic form, and it vanishes only for the zero signed measure.
Thus, two distinct minimizers cannot exist.  If one exponent is
\(2\), this strictness is lost in that coordinate.  The remaining minimizers
are then described by two simpler conditions: the absolute values must be
coupled antimonotonically, and the signs must satisfy the stated moment or
conditional-mean constraint.

The final example illustrates failure of the boundary uniqueness criterion in
a simple case.

\begin{example}
\label{ex:q2-uniform}
For \(p=q=2\) and \(\mu=\nu=\U[-1,1]\), the diamond minimizer is not unique.
Indeed, let
\(R\sim\U[0,1]\), and conditionally on \(R=r\),
choose signs \(\varepsilon,\eta=\pm1\) according to
\begin{equation}
  \label{eq:q2-sign-construction}
  \mathbb P(\varepsilon=a,\eta=b\mid R=r)
  =
  \frac{1+ab(r-1/2)}{4},\qquad a,b=\pm1.
\end{equation}
The random variables \(X=\varepsilon R\) and \(Y=\eta(1-R)\) have a minimizing law in
\(\Pi(\U[-1,1],\U[-1,1])\), but this law is not the diamond coupling.
\end{example}

\section{Diamond transport in DOT}\label{sec:dot}

 The DOT framework, introduced by
\cite{liu2023distorted}, replaces the linear expectation in the usual
Kantorovich problem by a distorted expectation.  The latter is the Choquet
integral with respect to a distorted probability capacity.  Distorted or nonadditive expectations
entered decision theory through, for example, anticipated and rank-dependent
utility \cite{quiggin1982anticipated,yaari1987dual} and Schmeidler's Choquet
integral representation and expected utility model
\cite{schmeidler1986integral,schmeidler1989subjective}. 
They are also widely used in mathematical finance and risk management \cite{acerbi2002spectral,mcneil2015quantitative} through modeling distortion and tail risk measures \cite{liu2021tailrisk}.

A distortion function is a nondecreasing function \(g:[0,1]\to[0,1]\) with
\(g(0)=0\) and \(g(1)=1\).  For a real-valued random variable \(Z\), its
distorted expectation is
\[
  \mathcal E^g[Z]
  :=
  \int_0^\infty g(\mathbb P(Z>t))\,\mathrm{d}t
  +
  \int_{-\infty}^0 ( g(\mathbb P(Z>t))-1)\,\mathrm{d}t,
\]
whenever the expression is well defined. 
If \(Z\ge0\), the second integral is
zero. 
As a convenient feature of distorted expectations, it is well known (e.g., \cite[Theorem 3]{wangWeiWillmot2020}) that the following equivalences hold:
\begin{equation}
    \label{eq:equiv-DOT}
   \mbox{$g$ is convex}
\iff \mbox{$Z\mapsto \mathcal E^g[Z]$ is concave} \iff  \mbox{$\mu \mapsto \mathcal E^g[Z_\mu]$ is convex, where $Z_\mu\sim\mu$.} 
\end{equation} 
The relationship in \eqref{eq:equiv-DOT} does not require $g$ to be nondecreasing or nonnegative, as studied in \cite{wangWeiWillmot2020}, so \eqref{eq:equiv-DOT} implies a corresponding version for concave $g$. 
For a Borel
transport cost \(c:\R^2\to \R\cup\{\infty\}\), the DOT problem is
\[
  \text{minimize}\quad
  \mathcal D_{g,c}(\pi)
  :=
  \int_0^\infty g\bigl(\pi(c>t)\bigr)\,\mathrm{d}t +  \int_{-\infty}^0( g\bigl(\pi(c>t)\bigr)-1)\,\mathrm{d}t 
  \quad\text{over }\pi\in\Pi(\mu,\nu).
\]
In other words, $  \mathcal D_{g,c}(\pi)=  \mathcal E^{g}[c(X,Y)]$, where $(X,Y)\sim \pi$.
Unlike QOT, this objective depends only on the distribution of the one-sample
cost \(c(X,Y)\).  Nevertheless, quadratic distortions produce QOT functionals,
and this is essentially the only unrestricted overlap, as shown in the following result.
In what follows, we always assume that $\mathcal D_{g,c}(\pi)$ is well defined on $\Pi(\mu,\nu)$.

\begin{proposition}
\label{prop:convex?}
For fixed marginals \(\mu,\nu\) and Borel cost \(c:\R^2\to\R\), if \(g\) is convex, then the DOT objective
 $D_{g,c}$ is convex.  Conversely,
if \(\mu\) and \(\nu\) are atomless, $g\in C^2[0,1]$, and $c(x,y)=x+y$, then  $ D_{g,c} $ is convex if and only if \(g\) is
convex.
\end{proposition}

\begin{proof}
If \(g\) is convex, then the distorted expectation is convex in the distribution, as explained in \eqref{eq:equiv-DOT}.  Since
\(\pi\mapsto\law_\pi(c(X,Y))\) is affine in \(\pi\), the map
\(\pi\mapsto\mathcal D_{g,c}(\pi)\) is convex on every \(\Pi(\mu,\nu)\).

We now prove the converse. Fix
\(x\in(0,1)\), and choose \(0<\varepsilon<\min\{x,1-x\}\). Put $q=x-\varepsilon$. 
Since $\mu,\nu$ are atomless, there exist sorted intervals $ L_\mu<A_0<A_1<H_\mu~\text{and}
  ~
  L_\nu<B_0<B_1<H_\nu$, both of which
 partition $\R$, with $\mu(H_\mu)=\nu(H_\nu)=q~\text{and}
  ~
  \mu(A_0)=\mu(A_1)=\nu(B_0)=\nu(B_1)=\varepsilon$.

Let $\pi\in\Pi(\mu,\nu)$ be such that \[
  \pi(H_\mu\times H_\nu)=q\qquad\text{and}
  \qquad
  \pi(L_\mu\times L_\nu)=1-q-2\varepsilon.
\] For $i,j\in\{0,1\}$, let \(\gamma_{ij}\) be the antimonotone couplings between
\(A_i\) and \(B_j\), each of total mass \(\varepsilon\). For
\(0\le\alpha\le1\), define
\[
  \pi_\alpha
  :=\pi|_{H_\mu\times H_\nu}+\pi|_{L_\mu\times L_\nu}+
  \alpha\gamma_{00}
  +(1-\alpha)\gamma_{01}
  +(1-\alpha)\gamma_{10}
  +\alpha\gamma_{11}.
\]
It follows that \(\pi_\alpha\in\Pi(\mu,\nu)\), and \(\alpha\mapsto\pi_\alpha\) is affine.

Set $F_\varepsilon(\alpha):=D_{g,c}(\pi_\alpha)$ and $T_{\varepsilon,\alpha}(t):=\pi_\alpha(X+Y>t)$. Since \(D_{g,c}\) is convex and \(\pi_\alpha\) is affine,
\(F_\varepsilon\) is convex. Moreover, \(T_{\varepsilon,\alpha}(t)\) is affine
in \(\alpha\). Therefore, for \(0<\alpha<1\),
\[
  F_\varepsilon''(\alpha)
  =
  \int_{\mathbb R}
  g''(T_{\varepsilon,\alpha}(t))
  \bigl(\partial_\alpha T_{\varepsilon,\alpha}(t)\bigr)^2\,\dd t
  \ge0.
\]

By our construction, 
\begin{align}
\partial_\alpha T_{\varepsilon,\alpha}(t)\neq 0\implies    x-\varepsilon
  =
  q
  \le
  T_{\varepsilon,\alpha}(t)
  \le
  q+2\varepsilon
  =
  x+\varepsilon.\label{eq:x+-ee}
\end{align}
Also $\int_{\R}
  \bigl(\partial_\alpha T_{\varepsilon,\alpha}(t)\bigr)^2\,\dd t
  >0,$ 
because $\pi_0$ and $\pi_1$ have
different \((X+Y)\)-tail functions on a set of positive Lebesgue measure due to the antimonotonicity. We obtain
\[
  \frac{
  \int_{\R}
  g''(T_{\varepsilon,\alpha}(t))
  \bigl(\partial_\alpha T_{\varepsilon,\alpha}(t)\bigr)^2\,\dd t
  }{
  \int_{\R}
  \bigl(\partial_\alpha T_{\varepsilon,\alpha}(t)\bigr)^2\,\dd t
  }\geq 0.
\]
The left-hand side is a weighted average of \(g''\) over
\([x-\varepsilon,x+\varepsilon]\) by \eqref{eq:x+-ee}. Thus, letting \(\varepsilon\downarrow0\) gives \(g''(x)\ge0\). Since
\(x\in(0,1)\) was arbitrary, \(g\) is convex on \([0,1]\).
\end{proof}

The cost $c(x,y)=x+y$ in the equivalence statement of Proposition \ref{prop:convex?} is chosen for simplicity, but it cannot be replaced by an arbitrary cost. For instance, if $c(x,y)=x$, then $D_{g,c}$ is constant on $\Pi(\mu,\nu)$ for every $g$ because $\mathcal E^g[c(X,Y)]$ only depends on $X$. The next example illustrates that the atomless assumption in that statement is also needed.
\begin{example}  Consider the nondegenerate marginals 
\[
  \mu=\nu=\frac{1}{2}\delta_0+\frac{1}{2}\delta_1, 
\]
and let $c(x,y)=x+y$.
   Every \(\pi\in\Pi(\mu,\nu)\) is determined
by \(\lambda=\pi(\{(1,1)\})\in[0,1/2]\), and then
\[
  \pi(X+Y=0)=\lambda,\qquad
  \pi(X+Y=1)=1-2\lambda,\qquad
  \pi(X+Y=2)=\lambda .
\]
Hence,
 $
  \mathcal D_{g,c}(\pi)=g(1-\lambda)+g(\lambda).
 $
 For any nonconvex $g$ with
 $
  g(1-\lambda)+g(\lambda)=1$ for  $\lambda\in[0,1],
 $ the map \(\pi\mapsto\mathcal D_{g,c}(\pi)\) is constant, hence convex,
on \(\Pi(\mu,\nu)\).
\end{example}

The  observation  in Proposition \ref{prop:convex?} separates DOT sharply from the other OT formulations.  In classical OT, the objective is affine in \(\pi\), so convexity is
automatic but uniqueness is limited by exposed faces of the transport polytope;
this is why the diamond transport cannot be isolated from its diamond-type
companions by a linear objective.  In DOT, convexity is controlled entirely by
the  distortion function \(g\), independently of any positive-definite
structure of the cost.  In QOT, by contrast, the convexity of
\(\pi\mapsto\mathcal Q_c(\pi)\) is a kernel property: it requires positive
definiteness of \(((x,y),(x',y'))\mapsto c(x,y,x',y')\) \cite[Proposition 1]{wangzhangQOT}.  This condition is much
more delicate than convexity of a scalar distortion and is one reason QOT can
be non-convex even for natural-looking costs.

In many  applications of distorted expectations, such as those in risk management \cite{mcneil2015quantitative}, 
the distortion function $g$ is often concave instead of convex, corresponding to the convexity of the functional $Z\mapsto \mathcal E^{g}[Z]$; see \eqref{eq:equiv-DOT} with sign flipped.  
Proposition \ref{prop:convex?} applies in such situations, by replacing convexity with concavity in all statements. 

The next result shows that DOT and QOT are two genuinely different frameworks, but they have  some common objects,  governed by a one-parameter family of distortion functions.

\begin{theorem}\label{thm:dot-qot-overlap}
Let \(g\) be a distortion function,   \(c:\R^2\to\R\) be
measurable and nonconstant, and  
\(h: \R^ 4\to\R\) be  measurable.   
The following statements are equivalent.
\begin{enumerate}
\item For every pair of marginals \(\mu,\nu\), whenever well defined,
\[
  \mathcal Q_h(\pi)=\mathcal D_{g,c}(\pi),
  \qquad \pi\in\Pi(\mu,\nu).
\]
\item There exists \(\alpha\in[-1,1]\) such that
\begin{align}
    g(u)=\alpha u^2+(1-\alpha)u,\qquad u\in [0,1],\label{eq:gu}
\end{align}
and, for \(z=(x,y)\) and \(z'=(x',y')\),
\[
  \frac{h(z,z')+h(z',z)}{2}
  =
  \alpha\min\{c(z),c(z')\}
  +\frac{1-\alpha}{2} ( c(z)+c(z')).
\]
\end{enumerate}
\end{theorem}

\begin{proof}
First, suppose that (ii) holds.  Fix \(\mu,\nu\) and
\(\pi\in\Pi(\mu,\nu)\).  Let \(W=(X,Y)\sim\pi\), and let \(W'\) be an
independent copy of \(W\).  Put \(Z=c(W)\) and \(Z'=c(W')\).  Note that 
\begin{align}
    \E[\min\{Z,Z'\}]
  =
  \int_0^\infty \pi(c>t)^2\,\dd t
  +
  \int_{-\infty}^0 (\pi(c>t)^2-1)\,\dd t ,\label{eq:emin}
\end{align}
because \(\mathbb P(\min\{Z,Z'\}>t)=\pi(c>t)^2\).
Let \(h^*\)  be given by $h^*(z,z')=(h(z,z')+h(z',z))/2$.  By \eqref{eq:gu} and \eqref{eq:emin},
\[
\begin{aligned}
  \mathcal D_{g,c}(\pi)
  &=
  \alpha \E[\min\{c(W),c(W')\}]
  +(1-\alpha)\E[c(W)]\\
  &=
  \iint
  \left(
  \alpha\min\{c(z),c(z')\}
  +\frac{1-\alpha}{2}(c(z)+c(z'))
  \right)
  \pi(\dd z)\pi(\dd z')
  =
  \mathcal Q_{h^*}(\pi).
\end{aligned}
\]
Since \(\pi\otimes\pi\) is symmetric and the symmetric part of \(h\) is \({h^*}\),
\(\mathcal Q_h(\pi)=\mathcal Q_{h^*}(\pi)\).  This proves (i).

Conversely, assume (i).  Choose \(z_0,z_1\in\R^2\) such that
\(a=c(z_0)<c(z_1)=b\), and put
\[
  \pi_\lambda=(1-\lambda)\delta_{z_0}+\lambda\delta_{z_1},
  \qquad 0\le\lambda\le1.
\]
If \(z_i=(x_i,y_i)\), then
\(\pi_\lambda\in\Pi((1-\lambda)\delta_{x_0}+\lambda\delta_{x_1},
(1-\lambda)\delta_{y_0}+\lambda\delta_{y_1})\).  A direct computation
gives
\[
  \mathcal D_{g,c}(\pi_\lambda)=a+(b-a)g(\lambda).
\]
On the other hand,
\[
  \mathcal Q_h(\pi_\lambda)
  =
  (1-\lambda)^2h(z_0,z_0)
  +\lambda(1-\lambda)( h(z_0,z_1)+h(z_1,z_0))
  +\lambda^2h(z_1,z_1),
\]
which is a quadratic polynomial in \(\lambda\).  Thus \(g\) is a quadratic
polynomial on \([0,1]\).  Since \(g(0)=0\) and \(g(1)=1\), it has the form $g(u)=\alpha u^2+(1-\alpha)u$. 
The monotonicity of the distortion function gives \(\alpha\in[-1,1]\). This proves \eqref{eq:gu}. 

Let \(h^*\) be as above.  The first
part of the proof shows that \(\mathcal Q_{h^*}(\pi)=\mathcal D_{g,c}(\pi)\) for
every coupling \(\pi\).  Hence
\[
  \iint (h(z,z')-{h^*}(z,z'))\,\pi(\dd z)\pi(\dd z')=0
  \quad\text{for every coupling }\pi.
\]
Taking \(\pi=\delta_z\) gives \(h(z,z)={h^*}(z,z)\).  Taking
\(\pi=(\delta_z+\delta_{z'})/2\) and using the symmetry of \({h^*}\) gives
\[
  h(z,z')+h(z',z)=2{h^*}(z,z').
\]
Thus the symmetric part of \(h\) is \({h^*}\), proving (ii).
\end{proof}

Theorem \ref{thm:dot-qot-overlap} illustrates that DOT and QOT coincide precisely when $g$ is a quadratic function.  When $\alpha\ge 0$, such $g$ is convex; when $\alpha\le 0$, such $g$ is concave. Thus, in combination with Proposition \ref{prop:convex?}, this result also highlights simple cases of QOT objectives that are either convex or concave.

We next return to the diamond structure.  The following result is a distorted
counterpart of Theorem~\ref{thm:absolute-value-classic}.  It uses the standard
fact that, for a concave distortion \(g\), \(\mathcal E^g\) is increasing with
respect to increasing convex order; see \cite[Lemma~2]{liu2023distorted}.
Indeed, this gives yet another  condition equivalent to those in \eqref{eq:equiv-DOT}.

\begin{theorem}\label{thm:dot-diamond-type}
Let \(\mu,\nu\in\Pcsym(\R)\).  Let
\(f:[0,\infty)^2\to[0,\infty)\) be  nondecreasing 
and supermodular.  If \(g\) is a concave distortion function, then every
diamond-type transport minimizes
\[
  \mathcal D_{g,c_f}(\pi),
  \qquad
  c_f(x,y):=f(|x|,|y|),
\]
over \(\Pi(\mu,\nu)\).
If further, $g$ is strictly concave, and either $f(r,s)=r+s$ or $f$ is strictly supermodular, then every minimizer of $\mathcal D_{g,c_f}(\pi)$ is diamond-type. 
\end{theorem}

\begin{proof}
Let \(\pi\in\Pi(\mu,\nu)\), and let \(\pi_\diamond\) be diamond-type.  Recall that \(\mu_{\mathrm{abs}}=|\cdot|_\#\mu\) and \(\nu_{\mathrm{abs}}=|\cdot|_\#\nu\). Write
\((R,S)\) for the absolute values under \(\pi\) and
\((R^-,S^-)\sim\gamma_{\rm ant}\) for the antimonotone coupling of
\(\mu_{\mathrm{abs}}\) and \(\nu_{\mathrm{abs}}\).  Since \(f\) is coordinatewise nondecreasing and
supermodular, \(\varphi\circ f\) is supermodular for every nondecreasing convex
\(\varphi\).  Indeed, this is the same four-point argument used for monotone
supermodular costs in distorted transport: the supermodular increment of
\(f\), together with coordinatewise monotonicity, lets convexity of
\(\varphi\) preserve the four-point inequality.

The rearrangement inequality therefore gives
\[
  \E[\varphi(f(R^-,S^-))]
  \le
  \E[\varphi(f(R,S))]
\]
for every nondecreasing convex \(\varphi\).  Hence
\begin{align}
     f(R^-,S^-)\le_{\rm icx} f(R,S).\label{eq:frs}
\end{align}
All diamond-type transports have the same absolute-value coupling
\(\gamma_{\rm ant}\), so the left-hand side of \eqref{eq:frs} is the cost distribution under
\(\pi_\diamond\).  Since \(g\) is concave, \(\mathcal E^g\) is increasing in
increasing convex order \cite[Lemma~2]{liu2023distorted}, and thus
\[
  \mathcal D_{g,c_f}(\pi_\diamond)
  \le
  \mathcal D_{g,c_f}(\pi).
\]
This proves the optimality of diamond-type transports.
For the last statement,
the case that $g$ is strictly concave and $f$ is strictly supermodular follows from  similar arguments to the proof of Theorem \ref{thm:absolute-value-classic}. 
The case that $g$ is strictly concave and $f(r,s)=r+s$ follows from the fact that  $\mathcal E^{g}$ is strictly increasing in increasing convex order for such $g$ (see \cite[Lemma 2]{lauzier2026risk}), so a minimizer $\pi$ must satisfy \((|x|,|y|)_\#\pi=\gamma_{\rm ant}\).  
\end{proof}

Concavity of $g$ corresponds to convexity of $\mathcal E^g$ as in \eqref{eq:equiv-DOT}, which is a common object in risk management.  
Indeed, $\mathcal E^{g}$ forms a large class of coherent risk measures in the sense of \cite{artzner1999coherent},  called spectral risk measures \cite{acerbi2002spectral}. 
This class includes the most important coherent risk measure, Expected Shortfall (ES, also called CVaR),
 widely used in financial regulation \cite{mcneil2015quantitative, wang2021axiomatic}  and optimization \cite{rockafellar2000optimization,zhu2009worst,chen2011tight}. 
 Therefore, Theorem \ref{thm:dot-diamond-type} gives the solutions to a class of risk minimization problems with given marginals  \cite[Chapter 8]{mcneil2015quantitative}, 
  $$
\mathcal E^{g}[f(|X|,|Y|)],
  $$
  where the bivariate risks are aggregated in the form of $f(|x|,|y|)$ with $f$ nondecreasing and supermodular, and evaluated by a spectral risk measure $\mathcal E^g$, including   ES as a special case.

The previous theorem gives a large family of DOT problems for which all
diamond-type transports are optimal.  The full diamond transport can also be
the unique DOT optimizer.  This is shown in the next example, where the
cost is finite only on the diamond boundary, and a monotone transform calibrates
the first variation of the distorted objective.

\begin{example}\label{thm:dot-unique-diamond}
We provide a DOT problem where the diamond transport is the unique minimizer. 
Let \(\mu=\nu=\U[-1,1]\), let \(p>1\), and set \(g_p(u)=u^p\).  On the diamond
boundary \(|x|+|y|=1\), define
\[
  B(x,y)
  =
  |x|
  +\mathbf 1_{\{x>0,y<0\}}
  +2\mathbf 1_{\{x<0,y>0\}}
  +3\mathbf 1_{\{x<0,y<0\}}.
\]
For \(0\le q<4\), put
\[
  \phi_p(q)
  =
  \int_0^q \left(1-\frac u4\right)^{1-p}\,\mathrm{d}u,
\]
with \(\phi_p(4)\) defined by the finite limit if \(1<p<2\), and by
\(+\infty\) if \(p\ge2\).  Define 
\[
  c_p(x,y)
  =
  \begin{cases}
  \phi_p(B(x,y)), & |x|+|y|=1,\\
  +\infty, & |x|+|y|\ne1.
  \end{cases}
\]
Then the diamond transport \(\pi_{\rm dia}\) is the unique minimizer of
\[
  \mathcal D_{g_p,c_p}(\pi)
  =
  \int_0^\infty \pi(c_p>t)^p\,\mathrm{d}t
\]
over \(\Pi(\U[-1,1],\U[-1,1])\).
\end{example}

\begin{proof}
Under \(\pi_{\rm dia}\), write
\[
  X=\varepsilon R,\qquad Y=\eta(1-R),
\]
where \(R\sim\U[0,1]\) and \(\varepsilon,\eta\) are independent Rademacher
signs.  The four sign pairs have probability \(1/4\), and \(R\) is uniform on
each side of the diamond.  Hence \(B(X,Y)\sim\U[0,4]\).  If
\[
  S_0(t):=\pi_{\rm dia}(c_p>t),
\]
then
\[
  S_0(\phi_p(q))=1-\frac q4,\qquad 0\le q<4.
\]
Therefore
\[
  \mathcal D_{g_p,c_p}(\pi_{\rm dia})
  =
  \int_0^4 \left(1-\frac q4\right)^p\phi_p'(q)\,\mathrm{d}q
  =
  \int_0^4 \left(1-\frac q4\right)\,\mathrm{d}q
  =
  2.
\]

Let \(\pi\in\Pi(\U[-1,1],\U[-1,1])\).  If
\(\pi(|X|+|Y|\ne1)>0\), then \(\mathcal D_{g_p,c_p}(\pi)=+\infty\), so any
minimizer must be supported on the diamond boundary.  For such a coupling,
\[
  X=\varepsilon R,\qquad Y=\eta(1-R),\qquad R\sim\U[0,1].
\]
The uniform marginals force
\[
  \mathbb P(\varepsilon=1\mid R=r)
  =
  \mathbb P(\eta=1\mid R=r)
  =
  1/2
  \quad\text{for a.e. }r.
\]
Thus, for some measurable \(k:[0,1]\to[-1,1]\),
\[
  \mathbb P(\varepsilon=a,\eta=b\mid R=r)
  =
  \frac{1+ab\,k(r)}4,\qquad a,b=\pm1.
\]
The diamond transport corresponds to \(k=0\) a.e.

Let \(S_\pi(t):=\pi(c_p>t)\).  Strict convexity of \(u\mapsto u^p\) gives
\[
  S_\pi(t)^p-S_0(t)^p
  \ge
  pS_0(t)^{p-1}(S_\pi(t)-S_0(t)),
\]
with equality only when \(S_\pi(t)=S_0(t)\).  Put
\[
  H(t):=\int_0^t pS_0(s)^{p-1}\,\mathrm{d}s.
\]
By the definition of \(\phi_p\),
\[
  H(\phi_p(q))
  =
  \int_0^q p\left(1-\frac u4\right)^{p-1}\phi_p'(u)\,\mathrm{d}u
  =
  pq.
\]
Hence \(H(c_p)=pB\) on the diamond boundary.  It follows that
\[
\begin{aligned}
  \mathcal D_{g_p,c_p}(\pi)-\mathcal D_{g_p,c_p}(\pi_{\rm dia})
  &\ge
  \int_0^\infty pS_0(t)^{p-1}(S_\pi(t)-S_0(t))\,\mathrm{d}t  \\
  &=
  \int H(c_p)\,\mathrm{d}(\pi-\pi_{\rm dia})
  =
  p\int B\,\mathrm{d}(\pi-\pi_{\rm dia}).
\end{aligned}
\]
Conditionally on \(R=r\), a direct computation yields
\[
 \E_\pi[B\mid R=r]=r+ \frac{0(1+k(r))+1(1-k(r))+2(1-k(r))+3(1+k(r))}{4}
  =
 r+ \frac32,
\]
which is the same conditional mean as under
\(\pi_{\rm dia}\).  Therefore, \(\int B\,\mathrm{d}\pi=\int B\,\mathrm{d}\pi_{\rm dia}\), and
\(\pi_{\rm dia}\) is optimal.

If equality holds, strict convexity forces \(S_\pi(t)=S_0(t)\) for almost
every \(t\).  Since both survival functions are monotone, \(c_p\) has the same
law under \(\pi\) and \(\pi_{\rm dia}\).  As \(\phi_p\) is strictly increasing,
\(B\) has the same law under the two couplings.  On the interval
\(B\in(0,1)\), only the side \(x>0,y>0\) contributes, and \(B=r\).  Hence, for
every Borel set \(A\subseteq(0,1)\),
\[
  \pi(B\in A)=\frac14\int_A(1+k(r))\,\mathrm{d}r,
  \qquad
  \pi_{\rm dia}(B\in A)=\frac{|A|}{4}.
\]
Equality of the \(B\)-laws gives \(\int_A k(r)\,\mathrm{d}r=0\) for every
Borel \(A\subseteq(0,1)\), so \(k=0\) a.e.  Therefore
\(\pi=\pi_{\rm dia}\), proving uniqueness.
\end{proof}

\section{A discussion on QAP}
\label{sec:QAP}
In this section, we briefly discuss the discrete analogue of QOT, which is the quadratic assignment problem (QAP).
We   present an interesting question here and describe a conjecture on its solutions, but we provide no results.

As we have seen in Section \ref{sec:main-results}, the diamond  transport minimizes several classes of problems in QOT, but such a form seems not to have been
studied in the QAP literature.  
 The QAP counterpart with cost
\(c\) is the minimization over permutations \(\sigma\) of \(\{1,2,\ldots,2m\}\) of
\begin{equation}
\label{eq:qap-objective}
  Q_{m,c}(\sigma):=\sum_{i=1}^{2m}\sum_{j=1}^{2m}
c(i,{\sigma(i)},j,{\sigma(j)}).
\end{equation}
In particular, this is a discrete matching problem between $\{1,2,\ldots,2m\}$ and itself.

The difficulty is that the continuous diamond optimizer is usually not
Monge, so a finite permutation must encode the diamond through a
combinatorial sign pattern.

Define the two
finite diamond permutations by
\begin{equation}
\label{eq:finite-diamond-permutation}
  \sigma_\varepsilon(i)
  :=
  m+1/2
  +\varepsilon(-1)^i
  (m-|i-m-1/2|),
  \qquad
  i=1,\ldots,2m,\quad \varepsilon\in\{-1,1\}.
\end{equation}
Equivalently, the points \((i,\sigma_\varepsilon(i))\) lie on the centered
discrete diamond:
\[
  |i-(m+1/2)|
  +
  |\sigma_\varepsilon(i)-(m+1/2)|
  =m.
\]
The two choices of \(\varepsilon\) are reflections of each other across the
horizontal center line \(\sigma_\varepsilon(i)=m+1/2\). See Figure \ref{fig:finite-diamond-permutation}.

\begin{figure}[t]
\centering
\setlength{\unitlength}{0.92mm}
\begin{picture}(104,96)
\thinlines
\put(16,16){\framebox(64,64){}}
\put(16,48){\line(1,0){64}}
\put(48,16){\line(0,1){64}}
\put(48,80){\line(1,-1){32}}
\put(80,48){\line(-1,-1){32}}
\put(48,16){\line(-1,1){32}}
\put(16,48){\line(1,1){32}}
\put(48,86){\makebox(0,0){\(16\)}}
\put(48,10){\makebox(0,0){\(1\)}}
\put(10,48){\makebox(0,0){\(1\)}}
\put(86,48){\makebox(0,0){\(16\)}}
\put(94,48){\makebox(0,0)[l]{\(i\)}}
\put(48,92){\makebox(0,0){\(\sigma_\varepsilon(i)\)}}
\thicklines
\put(18,46){\circle*{2.3}}
\put(22,54){\circle*{2.3}}
\put(26,38){\circle*{2.3}}
\put(30,62){\circle*{2.3}}
\put(34,30){\circle*{2.3}}
\put(38,70){\circle*{2.3}}
\put(42,22){\circle*{2.3}}
\put(46,78){\circle*{2.3}}
\put(50,18){\circle*{2.3}}
\put(54,74){\circle*{2.3}}
\put(58,26){\circle*{2.3}}
\put(62,66){\circle*{2.3}}
\put(66,34){\circle*{2.3}}
\put(70,58){\circle*{2.3}}
\put(74,42){\circle*{2.3}}
\put(78,50){\circle*{2.3}}
\end{picture}
\caption{The finite diamond permutation for \(m=8\) and \(\varepsilon=1\).
The plotted points are \((i,\sigma_\varepsilon(i))\), \(i\in \{1,\dots,16\}\), and lie on
the centered discrete diamond $|i-17/2|+|\sigma_\varepsilon(i)-17/2|=8.$ 
The other finite diamond permutation is obtained by reflection across the
horizontal center line.}
\label{fig:finite-diamond-permutation}
\end{figure}
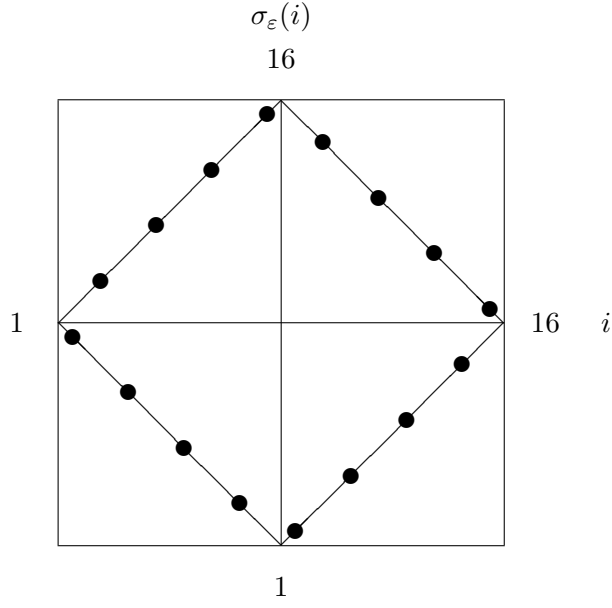

\begin{problem}
\label{conj}
 For the rectangular power cost
\[
  c_{p,q}(x,y,x',y')=|x-x'|^p|y-y'|^q,\qquad 1<p,q\le2,
\]
determine the complete set of minimizers of \(Q_{m,c_{p,q}}\) over all
permutations \(\sigma\) of \(\{1,2,\ldots,2m\}\).  In particular, decide for which
\((m,p,q)\) the only minimizers are the two finite diamond permutations
\(\sigma_\varepsilon\) in \eqref{eq:finite-diamond-permutation}.
We conjecture that the two finite diamond permutations are the unique minimizers for very general choices of $(m,p,q)$. 
\end{problem}

The simplest non-trivial case $m=2$ in Problem \ref{conj} follows from a direct enumeration, and the cases $m=3,4,5$ can be proved using computer-assisted verification. The case $p=q$ was briefly discussed in \cite[Appendix D(ii)]{wangzhangQOT} with no explicit conjectural solution.

Problem \ref{conj} belongs to a classical line of work on explicitly solvable
QAPs.  Examples include QAPs with Monge or
anti-Monge structure \cite{burkard1998antimongeToeplitz}, Toeplitz structure
\cite{burkard1998antimongeToeplitz}, circulant structure
\cite{deinekoWoeginger1998solvable}, Kalmanson matrices
\cite{deinekoWoeginger1998solvable}, Robinsonian matrices
\cite{laurentSeminaroti2015robinsonian}, block-structured matrices
\cite{celaDeinekoWoeginger2015block}, and linearizable special cases
\cite{celaDeinekoWoeginger2016linearizable}.  To the best of our knowledge,
however, the diamond-type permutations considered here have not been isolated
in the QAP literature.  Their distinctive feature is that they discretize a
continuous optimizer which is typically not Monge, where the permutation must encode
the antimonotone matching of radii together with a combinatorial sign pattern.

\section{Proofs of results in Section \ref{sec:main-results}}\label{sec:proofs}
In this section, we provide detailed proofs of the QOT results and examples in
Section~\ref{sec:main-results}.  The proofs are organized in the same order as
the results.

\subsection{Proof of results in Section \ref{31}}\label{sec:pf1}

\begin{lemma}
\label{lem:gradient}
Let \(\mu,\nu\) be compactly supported probability measures on $\R$, let
\(S=\operatorname{supp}\mu\times\operatorname{supp}\nu\), and let
\(\pi_{\mathrm{dia}}\in\Pi(\mu,\nu)\).  Let
\(c:S\times S\to\R\) be bounded, continuous, and symmetric, and assume that the kernel 
\(((x,y),(x',y'))\mapsto c(x,y,x',y')\) is positive definite on \(S\).  Define
\begin{equation}
\label{eq:linearized-cost}
  \widetilde c(x,y):=
  \int c(x,y,x',y')\,\dd\pi_{\mathrm{dia}}(x',y').
\end{equation}
If \(\pi_{\mathrm{dia}}\) minimizes the classical OT problem
\begin{equation}
\label{eq:linearized-ot}
  \min_{\pi\in\Pi(\mu,\nu)}\int \widetilde c\,\dd\pi,
\end{equation}
then \(\pi_{\mathrm{dia}}\) minimizes \(\mathcal Q_c\) over
\(\Pi(\mu,\nu)\).  If the kernel is strictly positive definite on finite
signed measures supported on \(S\), then the QOT minimizer is unique.
\end{lemma}

\begin{proof}
This is the first-order convexity argument used in the proof of
\cite[Theorem 12]{wangzhangQOT}; positive definiteness gives convexity as in
\cite[Proposition 1]{wangzhangQOT}, and the optimality of
\(\pi_{\mathrm{dia}}\) for the linearized OT problem gives global optimality.
The same argument with strict positive definiteness gives uniqueness.
\end{proof}

Let \(S\subseteq\R^2\).  A function \(G:S\to\R\) has the \emph{diamond
modularity pattern} on \(S\) if \(G\) is supermodular on
\(S\cap[0,\infty)^2\) and \(S\cap(-\infty,0]^2\), and is submodular on
\(S\cap([0,\infty)\times(-\infty,0])\) and
\(S\cap((-\infty,0]\times[0,\infty))\).  In the continuous cases below,
inequalities involving axes are interpreted by continuity.

\begin{lemma}
\label{lem:sign}
In the same setup as Theorem~\ref{thm:main}, assume in addition that
\(\mu,\nu\) are atomless, and let \(\widetilde c\) be defined by
\eqref{eq:linearized-cost}.  Then \(\widetilde c\) has the diamond modularity
pattern on
\(S=\operatorname{supp}\mu\times\operatorname{supp}\nu\).
\end{lemma}

\begin{proof}
Write \((X',Y')=(\varepsilon A,\eta B)\) as in the diamond construction,
where \((A,B)\) are antimonotonically coupled positive radii and
\(\varepsilon,\eta\) are independent Rademacher random variables.  We prove
the four-point inequalities directly.  Let \(0\le x_1<x_2\) and
\(0\le y_1<y_2\).  Since the continuous extension of \(W\) is bounded, the
mixed rectangle increment of \(\widetilde c\) is
\begin{align}
&\widetilde c(x_2,y_2)+\widetilde c(x_1,y_1)
-\widetilde c(x_2,y_1)-\widetilde c(x_1,y_2)\nonumber\\
&\quad =
\E\!\left[
h(|x_2-X'|,|y_2-Y'|)
+h(|x_1-X'|,|y_1-Y'|)
\right.\nonumber\\
&\qquad\qquad\left.
-h(|x_2-X'|,|y_1-Y'|)
-h(|x_1-X'|,|y_2-Y'|)
\right]\nonumber\\
&\quad =
\int_{x_1}^{x_2}\int_{y_1}^{y_2}
\E\!\left[
\operatorname{sgn}(u-X')\operatorname{sgn}(v-Y')
W(|u-X'|,|v-Y'|)
\right]\,\dd v\,\dd u .\label{eq:integral}
\end{align}
The last equality is obtained as follows.  For fixed \(X'=x'\) and
\(Y'=y'\), split \([x_1,x_2]\times[y_1,y_2]\) by the lines \(u=x'\) and
\(v=y'\).  On each resulting subrectangle the signs of \(u-x'\) and \(v-y'\)
are constant, so the usual mixed-derivative identity applies to
\((u,v)\mapsto h(|u-x'|,|v-y'|)\).  Summing the subrectangles gives the
equality in \eqref{eq:integral}.  If a cutting line lies on the boundary, or a
subrectangle touches one of the lines, the same identity follows by moving the
line by \(\epsilon\) and then letting \(\epsilon\downarrow0\), using
continuity of \(h\) and the continuous extension of \(W\).  Thus, it suffices
to show that, for \(u,v\ge0\), the conditional average of the integrand in \eqref{eq:integral} given
\(A=a,B=b\) is nonnegative.  Set
\begin{equation}
\label{eq:delta-average}
  \Delta_{u,v}(a,b)
  =
  \frac14\sum_{\epsilon,\eta=\pm1}
  \operatorname{sgn}(u-\epsilon a)\operatorname{sgn}(v-\eta b)
  W(|u-\epsilon a|,|v-\eta b|).
\end{equation}
There are four cases.

\begin{itemize}
\item If \(u\ge a\) and \(v\ge b\), every sign is positive, so
\(\Delta_{u,v}(a,b)\ge0\).

\item If \(u<a\) and \(v\ge b\), then
\[
\begin{aligned}
4\Delta_{u,v}(a,b)
&=
-W(a-u,v-b)-W(a-u,v+b)\\
&\quad+W(a+u,v-b)+W(a+u,v+b)\ge0,
\end{aligned}
\]
by monotonicity of \(W\) in the first coordinate.

\item If \(u\ge a\) and \(v<b\), the same argument gives
\(\Delta_{u,v}(a,b)\ge0\), now using monotonicity of \(W\) in the second
coordinate.

\item If \(u<a\) and \(v<b\), then
\[
\begin{aligned}
4\Delta_{u,v}(a,b)
&=
W(a-u,b-v)-W(a-u,b+v)\\
&\quad-W(a+u,b-v)+W(a+u,b+v)\ge0,
\end{aligned}
\]
because \(a-u\le a+u\), \(b-v\le b+v\), and \(W\) is supermodular.
\end{itemize}
Hence
\(\widetilde c\) is supermodular in the first quadrant.  Since
\(\widetilde c\) is even in each coordinate, the same statement holds in the
third quadrant, while changing the sign of exactly one coordinate reverses
the mixed rectangle increment.  Thus, \(\widetilde c\) is submodular in the
second and fourth quadrants.  This proves the diamond modularity pattern.
\end{proof}

\begin{lemma}
\label{lem:quadrant-ot}
Let \(\mu,\nu\in\Pcsym(\R)\) be atomless, and put
\(S=\operatorname{supp}\mu\times\operatorname{supp}\nu\).  Let
\(\ell:S\to\R\) be continuous and even in each coordinate, and assume that it
has the diamond modularity pattern on \(S\).  Then the diamond transport minimizes the linear OT problem
\(\min_{\pi\in\Pi(\mu,\nu)}\int \ell(x,y)\,\dd\pi(x,y)\).
\end{lemma}

\begin{proof}
Let \(\pi\in\Pi(\mu,\nu)\), and let \((X,Y)\sim\pi\).  Let
\(\bar\pi\) be the law of \((\varepsilon X,\eta Y)\), where
\(\varepsilon,\eta\) are independent Rademacher signs, independent of
\((X,Y)\).  Since \(\mu,\nu\) are symmetric, \(\bar\pi\in\Pi(\mu,\nu)\), and
since \(\ell\) is even in each coordinate,
\[
  \int \ell\,\dd\bar\pi=\int \ell\,\dd\pi .
\]
Thus, it is enough to minimize over fourfold-symmetric couplings.  For such a
coupling, the objective is determined by the law \(\gamma\) of
\((R,T)=(|X|,|Y|)\): $\int \ell\,\dd\bar\pi
  =
  \int \ell\,\dd\gamma$, 
where \(\gamma\in\Pi(\mu_{\mathrm{abs}},\nu_{\mathrm{abs}})\).

By the diamond modularity pattern, \(\ell\) is supermodular on the first
quadrant.  The one-dimensional rearrangement theorem for supermodular
costs \cite[Chapter~3.1]{rachev1998mass} implies that
\[
  \int \ell\,\dd\gamma
  \ge
  \int \ell\,\dd\piant(\mu_{\mathrm{abs}},\nu_{\mathrm{abs}})
\]
for every \(\gamma\in\Pi(\mu_{\mathrm{abs}},\nu_{\mathrm{abs}})\).  The
fourfold-symmetric lift of this antimonotone radius coupling is exactly
\(\pi_{\mathrm{dia}}\).  Hence \(\pi_{\mathrm{dia}}\) minimizes the linear OT
problem.
\end{proof}

\begin{proof}[Proof of Theorem~\ref{thm:main}]
First assume that \(\mu,\nu\) are atomless.  By Lemma~\ref{lem:sign}, the
first variation \(\widetilde c\) has the diamond modularity pattern required
by Lemma~\ref{lem:quadrant-ot}.  Moreover, \(\widetilde c\) is continuous by
bounded convergence and is even in each coordinate by the independent sign
symmetry of \(\pi_{\mathrm{dia}}\).  Consequently, \(\pi_{\mathrm{dia}}\)
minimizes the linearized OT problem with cost \(\widetilde c\).  Lemma~\ref{lem:gradient}
upgrades this first-order optimality to global optimality.

Now let \(\mu,\nu\) be arbitrary compactly supported symmetric marginals.  If
one support diameter is zero, then \(\Pi(\mu,\nu)\) is a singleton and there is
nothing to prove.  Otherwise choose atomless symmetric probability measures
\(\mu_n,\nu_n\), supported respectively on \(\operatorname{conv}(\operatorname{supp}\mu)\) and \(\operatorname{conv}(\operatorname{supp}\nu)\), such that
\(\mu_n\Rightarrow\mu\) and \(\nu_n\Rightarrow\nu\).
Let \(\pi_{\mathrm{dia}}^n\) be the diamond transport for
\((\mu_n,\nu_n)\).  By the atomless case, \(\pi_{\mathrm{dia}}^n\) minimizes
\(\mathcal Q_c\) over \(\Pi(\mu_n,\nu_n)\).  Extend the restriction of \(c\)
from the compact set \((\operatorname{conv}(\operatorname{supp}\mu)\times \operatorname{conv}(\operatorname{supp}\nu))^2\) to a bounded continuous
function on \(\R^4\).  Since this extension does not change the objective on
the couplings under consideration, and since boundedness gives the uniform
integrability required in \cite[Proposition 3]{wangzhangQOT}, QOT stability
implies that every weak limit point of \((\pi_{\mathrm{dia}}^n)\) is a
minimizer for the limit marginals \((\mu,\nu)\).
Finally,
\(\mu_{n,\mathrm{abs}}\Rightarrow\mu_{\mathrm{abs}}\) and
\(\nu_{n,\mathrm{abs}}\Rightarrow\nu_{\mathrm{abs}}\).  The quantile
construction of the antimonotone radius coupling therefore gives
\(\pi_{\mathrm{dia}}^n\Rightarrow\pi_{\mathrm{dia}}\).  Hence
\(\pi_{\mathrm{dia}}\) minimizes \(\mathcal Q_c\) over \(\Pi(\mu,\nu)\).

If the kernel is strictly positive definite on finite signed measures
supported on \(S\), uniqueness follows from strict convexity on
\(\Pi(\mu,\nu)\).  Indeed, if \(\pi_0,\pi_1\) are distinct minimizers and
\(\sigma=\pi_1-\pi_0\), then
\[
  \mathcal Q_c\!\left(\frac{\pi_0+\pi_1}{2}\right)
  =
  \frac12\mathcal Q_c(\pi_0)+\frac12\mathcal Q_c(\pi_1)
  -\frac14\iint c(z,z')\,\dd\sigma(z)\,\dd\sigma(z'),
\]
and the last integral is strictly positive, contradicting optimality of the
midpoint.
\end{proof}

\begin{proof}[Proof of Example~\ref{ex:weighted-kernels}]
(i) We use the shorthand notation $u=ar^2+bs^2$ and let $h(r,s)=(\beta+ar^2+bs^2)^{-\gamma}=(\beta+u)^{-\gamma}.$ 
Then
\[
  W(r,s):=\frac{\partial^2 h}{\partial r\,\partial s}(r,s)
  =
  4ab\gamma(\gamma+1)rs(\beta+u)^{-\gamma-2}\ge0 .
\]
Moreover,
\[
  \frac{\partial W}{\partial r}(r,s)
  =
  4ab\gamma(\gamma+1)s(\beta+u)^{-\gamma-3}
  \big(\beta+bs^2-(2\gamma+3)ar^2\big),
\]
and similarly
\[
  \frac{\partial W}{\partial s}(r,s)
  =
  4ab\gamma(\gamma+1)r(\beta+u)^{-\gamma-3}
  \big(\beta+ar^2-(2\gamma+3)bs^2\big).
\]
Thus, \(W\) is coordinatewise nondecreasing on
\([0,D_x]\times[0,D_y]\) under
\(\beta\ge(2\gamma+3)\max\{aD_x^2,bD_y^2\}\).  Finally,
\[
\begin{aligned}
  \frac{\partial^2 W}{\partial r\,\partial s}(r,s)
  &=
  4ab\gamma(\gamma+1)(\beta+u)^{-\gamma-4}  \\
  &\quad \times
  \left(
    (\beta+u)\big(\beta-(2\gamma+3)u\big)
    +4(\gamma+2)(\gamma+3)ab r^2s^2
  \right).
\end{aligned}
\]
The bracket is nonnegative under the same condition.  Indeed, put
\(K:=2\gamma+3\), \(A=ar^2\), \(B=bs^2\), and
\(M=\max\{aD_x^2,bD_y^2\}\).  Then \(0\le A,B\le M\) and
\(\beta\ge KM\).  For fixed \(A,B\), the function
\[
  \beta\mapsto (\beta+A+B)\bigl(\beta-K(A+B)\bigr)
\]
is nondecreasing on \([KM,\infty)\), since its derivative is
\[
  2\beta+(1-K)(A+B)
  \ge 2KM+2(1-K)M
  =2M\ge0 .
\]
Therefore
\[
\begin{aligned}
&(\beta+A+B)\bigl(\beta-K(A+B)\bigr)
  +4(\gamma+2)(\gamma+3)AB  \\
&\quad\ge
(KM+A+B)K(M-A-B)+4(\gamma+2)(\gamma+3)AB  \\
&\quad=
K^2(M-A)(M-B)
+KA(M-A)+KB(M-B)
+(2K+3)AB
\ge0 .
\end{aligned}
\]
Hence \(W\) is supermodular.  Positive definiteness follows from
\[
  (\beta+ar^2+bs^2)^{-\gamma}
  =
  \frac1{\Gamma(\gamma)}
  \int_0^\infty
  t^{\gamma-1}e^{-\beta t}e^{-atr^2}e^{-bts^2}\,\mathrm{d}t,
\]
which writes the kernel as a positive mixture of Gaussian positive definite
kernels.  The same representation gives a strictly positive spectral density
on \(\R^2\), so the kernel is strictly positive definite on finite signed
measures supported on \(\operatorname{supp}\mu\times\operatorname{supp}\nu\).
Theorem~\ref{thm:main} then proves (i).

(ii) Let $h(r,s)=e^{-ar^2-bs^2}.$ 
Then $W(r,s)=4abrs\,e^{-ar^2-bs^2}.$ 
A direct computation gives
\[
  \frac{\partial W}{\partial r}
  =
  4abs\,e^{-ar^2-bs^2}(1-2ar^2),
  \qquad
  \frac{\partial W}{\partial s}
  =
  4abr\,e^{-ar^2-bs^2}(1-2bs^2),
\]
and
\[
  \frac{\partial^2 W}{\partial r\,\partial s}
  =
  4ab e^{-ar^2-bs^2}(1-2ar^2)(1-2bs^2).
\]
Thus, \(W\) is nonnegative, coordinatewise nondecreasing, and supermodular on
\([0,D_x]\times[0,D_y]\) if
\(2aD_x^2\le1\) and \(2bD_y^2\le1\).  Positive definiteness and strict positive
definiteness follow from the strictly positive Gaussian spectral density.
Theorem~\ref{thm:main} then proves (ii).  
\end{proof}

\subsection{Proof of results in Section \ref{32}}

\begin{lemma}\label{lemma:abs integrability}
Let \(\pi_0,\pi_1\in\Pi(\mu,\nu)\) satisfy
\(\mathcal Q(\pi_0)+\mathcal Q(\pi_1)<\infty\), and set
\(\sigma:=\pi_1-\pi_0\). Under the assumptions of
Theorem~\ref{thm:taylor-products}(ii)--(iii),
\[
\iint
\psi_1(|x-x'|)\psi_2(|y-y'|)
\,\mathrm{d}|\sigma|(x,y)\,\mathrm{d}|\sigma|(x',y')<\infty .
\]
\end{lemma}

\begin{proof}
Since \(u\mapsto \psi_i(|u|)\) is continuous negative definite and vanishes at
\(0\), \(d_i(u,v):=\psi_i(|u-v|)^{1/2}\) is a pseudometric. Hence
\begin{align}
    \psi_i(|u-v|)\le 2\psi_i(|u|)+2\psi_i(|v|). \label{1}
\end{align}
We first claim that, if \(\pi\in\Pi(\mu,\nu)\) and \(\mathcal Q(\pi)<\infty\), then
\begin{align}
    \int \psi_1(|x|)\psi_2(|y|)\,\mathrm{d}\pi(x,y)<\infty . \label{2}
\end{align}
Choose \(R>0\) with \(m:=\pi([-R,R]^2)>0\), and put \(A_i:=\psi_i(R)\).
If \((x',y')\in[-R,R]^2\) and
\(\psi_i(|u|)>4A_i\), then the triangle inequality for \(d_i\) gives
\(\psi_i(|u-u'|)\ge \psi_i(|u|)/4\). Thus, on
\(G:=\{\psi_1(|x|)>4A_1,\ \psi_2(|y|)>4A_2\}\),
\[
\mathcal Q(\pi)\ge
\frac{m}{16}\int_G \psi_1(|x|)\psi_2(|y|)\,\mathrm{d}\pi(x,y).
\]
On \(G^c\),
\[
\psi_1(|x|)\psi_2(|y|)
\le 4A_1\psi_2(|y|)+4A_2\psi_1(|x|),
\]
which is integrable by the marginal moment assumptions. This proves \eqref{2}.

Apply \eqref{2} to \(\pi_0\) and \(\pi_1\), and put \(\theta:=\pi_0+\pi_1\). Then
\(\int\psi_1(|x|)\psi_2(|y|)\,\mathrm{d}\theta<\infty\). By \eqref{1},
\[
\psi_1(|x-x'|)\psi_2(|y-y'|)
\le
4(\psi_1(|x|)+\psi_1(|x'|))
(\psi_2(|y|)+\psi_2(|y'|)),
\]
and the right-hand side is integrable with respect to
\(\theta\otimes\theta\). Since \(|\sigma|\le\theta\), the result follows.
\end{proof}

\begin{proof}[Proof of Theorem~\ref{thm:taylor-products}]
We split the proof into five steps.

\smallskip\noindent\textbf{Step I: optimality for compactly supported marginals with an extra convexity condition.} 
First, we prove optimality assuming moreover that \(\mu,\nu\in\Pcsym(\R)\) with finite support diameters $D_1,D_2$ respectively, and \(\psi_i'\) has finite one-sided limits
at \(0\) and at \(D_i\), and satisfies the condition
\begin{align}
     \inf_{\substack{0<r<D_i\\ \psi_i'(r)>0}}
  \frac{\psi_i''(r)}{(\psi_i'(r))^2}>0.\label{eq:cond1}
\end{align}  For \(\alpha>0\), define
\begin{equation}
\label{eq:taylor-kernel}
  K_\alpha(x,y,x',y')
  =
  \exp(-\alpha\psi_1(|x-x'|)-\alpha\psi_2(|y-y'|)).
\end{equation}
Write
\[
  h_\alpha(r,s):=a_1(r)a_2(s),
  \qquad
  a_i(r):=e^{-\alpha\psi_i(r)},\quad i=1,2 .
\]
We check the hypotheses of Assumption \ref{ass:convex-diamond} for the cost
\(K_\alpha=h_\alpha(|x-x'|,|y-y'|)\), for all sufficiently small
\(\alpha>0\).

\begin{itemize}

\item \emph{Regularity.}
The function \(h_\alpha\) is continuous on
\([0,D_1]\times[0,D_2]\) and \(C^2\) on
\((0,D_1)\times(0,D_2)\).  Moreover,
\[
  a_i'(r)=-\alpha \psi_i'(r)e^{-\alpha\psi_i(r)},
\]
and the assumed finite one-sided limits of \(\psi_i'\) at \(0\) and \(D_i\)
imply that \(a_i'\), and hence
\[
  W_\alpha(r,s)
  :=
  \frac{\partial^2 h_\alpha}{\partial r\,\partial s}(r,s)
  =
  a_1'(r)a_2'(s),
\]
has a finite continuous extension to the closed rectangle.

\item \emph{Positive definiteness.}
By our positive definite assumption, each kernel
\((u,v)\mapsto e^{-\alpha\psi_i(|u-v|)}\) is positive definite on \(\R\).
Hence, their product \(K_\alpha\) is positive definite by the Schur product
theorem \cite[Theorem 3.1.12]{berg1984harmonic}.

\item \emph{Sign, monotonicity, and supermodularity of \(W_\alpha\).}
Since \(\psi_i\) is nondecreasing, \(a_i'\le0\).  Also, for
\(0<r<D_i\),
\[
  a_i''(r)
  =
  e^{-\alpha\psi_i(r)}
  \big(\alpha^2(\psi_i'(r))^2-\alpha\psi_i''(r)\big).
\]
By \eqref{eq:cond1}, choosing \(\alpha>0\) smaller than the lower bound for
\(\psi_i''/(\psi_i')^2\), uniformly for \(i=1,2\), gives \(a_i''\le0\).
When \(\psi_i'(r)=0\), convexity gives \(\psi_i''(r)\ge0\), so the same
inequality holds.  Therefore, \(a_i\) is nonincreasing and concave.  Since
\(W_\alpha=a_1'a_2'\), we have \(W_\alpha\ge0\), while
\[
  \partial_r W_\alpha=a_1''a_2'\ge0,
  \qquad
  \partial_s W_\alpha=a_1'a_2''\ge0,
\]
and the four-point supermodularity inequality follows from
\(a_1'',a_2''\le0\), equivalently from the fact that \(a_1'\) and \(a_2'\)
are nonincreasing and nonpositive.
\end{itemize}

Thus, for $\alpha>0$ small enough, Theorem~\ref{thm:main} applies to \(K_\alpha\), and
\begin{equation}
\label{eq:jalpha-minimization}
  \mathcal Q_\alpha(\pi_{\mathrm{dia}})
  \le
  \mathcal Q_\alpha(\pi),
  \qquad \pi\in\Pi(\mu,\nu),
\end{equation}
where \(\mathcal Q_\alpha\) is the quadratic-form energy with kernel
\eqref{eq:taylor-kernel}.

Write \(A=\psi_1(|x-x'|)\) and \(B=\psi_2(|y-y'|)\).  Uniformly on the compact
support,
\begin{equation}
\label{eq:taylor-expansion}
  e^{-\alpha(A+B)}
  =
  1-\alpha(A+B)
 +\frac{\alpha^2}{2}(A+B)^2
  +O(\alpha^3).
\end{equation}
The \(O(\alpha^3)\) bound is uniform over all pairs
\((x,y),(x',y')\) in the common compact support, because \(A+B\) is bounded.
After integration against any \(\pi\otimes\pi\), it remains \(O(\alpha^3)\)
with a constant independent of \(\pi\).
In the quadratic-form energy, the terms \(1\), \(A\), \(B\), \(A^2\), and \(B^2\) depend
only on the fixed marginals.  Substituting \eqref{eq:taylor-expansion} into
\eqref{eq:jalpha-minimization} gives, for every \(\pi\),
\begin{equation}
\label{eq:taylor-energy-gap}
\begin{aligned}
0
&\le
\mathcal Q_\alpha(\pi)-\mathcal Q_\alpha(\pi_{\mathrm{dia}})=
\alpha^2
\left(
\iint AB\,\dd\pi\,\dd\pi
-
\iint AB\,\dd\pi_{\mathrm{dia}}\,\dd\pi_{\mathrm{dia}}
\right)
 +O(\alpha^3).
\end{aligned}
\end{equation}
Dividing \eqref{eq:taylor-energy-gap} by \(\alpha^2\) and letting
\(\alpha\downarrow0\) proves that \(\pi_{\mathrm{dia}}\) minimizes the product
cost \eqref{eq:taylor-product-cost}.

\smallskip\noindent\textbf{Step II: optimality for compactly supported marginals by uniform convergence.} Now suppose only the extra constraint that \(\mu,\nu\) are compactly supported.  Fix
\(\rho\in(1,2)\) and set
\[
  \psi_{i,\varepsilon}(r):=\psi_i(r)+\varepsilon r^\rho .
\]
The Schoenberg condition is preserved by addition, since
\(u\mapsto |u|^\rho\) is continuous negative definite.  Also,
\[
  \psi_{i,\varepsilon}''(r)
  =
  \psi_i''(r)+\varepsilon\rho(\rho-1)r^{\rho-2}>0.
\]
On every compact interval \([0,D]\), \(\psi_{i,\varepsilon}'\) is bounded, and
hence \eqref{eq:cond1} is satisfied for \(\psi_{i,\varepsilon}\). 
Thus, Step I applies, and letting
\(\varepsilon\downarrow0\) gives the compactly supported result by uniform
convergence of the product costs on the compact distance rectangle.

\smallskip\noindent\textbf{Step III: optimality for the general case by truncation.} 
It remains only to remove the compactness assumption for the optimality result.  Let
\(T_n(x)=(-n)\vee(x\wedge n)\), and put
\(\mu_n=(T_n)_\#\mu\), \(\nu_n=(T_n)_\#\nu\).  These marginals are bounded and
symmetric.  For an arbitrary \(\pi\in\Pi(\mu,\nu)\), let
\(\pi_n=(T_n,T_n)_\#\pi\).  
Step II gives
\begin{equation}
\label{eq:truncated-min}
  \mathcal Q^{(n)}(\pi_n)\ge \mathcal Q^{(n)}(\pi_{\mathrm{dia}}^{(n)}),
\end{equation}
where \(\mathcal Q^{(n)}\) is the same quadratic-form energy functional for
\(\mu_n,\nu_n\), and \(\pi_{\mathrm{dia}}^{(n)}\) is their diamond transport.
Because \(T_n\) is odd and nondecreasing,
\(\pi_{\mathrm{dia}}^{(n)}=(T_n,T_n)_\#\pi_{\mathrm{dia}}\).\footnote{Indeed,
for a nonnegative random variable \(A\), the left-continuous quantile of
\(T_n(A)=A\wedge n\) is
\(F_{T_n(A)}^{-1}(u)=T_n(F_A^{-1}(u))\).  Thus, the pushforward of the
\(\piant\) coupling of the absolute values is the \(\piant\) coupling of the
pushed-forward absolute values, even though clipping may create atoms at
\(n\).}
Moreover, by the monotonicity for $\psi_i$,  for every fixed \(x,x',y,y'\),
\[
  \psi_1(|T_n(x)-T_n(x')|)\psi_2(|T_n(y)-T_n(y')|)
  \uparrow \psi_1(|x-x'|)\psi_2(|y-y'|) .
\]
Applying monotone convergence to \(\pi\otimes\pi\) and to
\(\pi_{\mathrm{dia}}\otimes\pi_{\mathrm{dia}}\) gives the desired convergence.
Indeed, since \(\pi_n=(T_n,T_n)_\#\pi\),
\[
\begin{aligned}
\mathcal Q^{(n)}(\pi_n)
&=
\iint \psi_1(|T_n(x)-T_n(x')|)\psi_2(|T_n(y)-T_n(y')|)
\,\dd\pi(x,y)\,\dd\pi(x',y'),\\
\mathcal Q(\pi)
&=
\iint \psi_1(|x-x'|)\psi_2(|y-y'|)
\,\dd\pi(x,y)\,\dd\pi(x',y').
\end{aligned}
\]
The same identities hold with \(\pi_{\mathrm{dia}}\) in place of \(\pi\),
because \(\pi_{\mathrm{dia}}^{(n)}=(T_n,T_n)_\#\pi_{\mathrm{dia}}\).  Hence
\begin{equation}
\label{eq:truncated-convergence}
  \mathcal Q^{(n)}(\pi_n)\uparrow \mathcal Q(\pi),
  \qquad
  \mathcal Q^{(n)}(\pi_{\mathrm{dia}}^{(n)})\uparrow \mathcal Q(\pi_{\mathrm{dia}}).
\end{equation}
Passing \eqref{eq:truncated-min} to the limit using
\eqref{eq:truncated-convergence} proves
\(\mathcal Q(\pi)\ge \mathcal Q(\pi_{\mathrm{dia}})\).

\smallskip\noindent\textbf{Step IV: finiteness of $\mathcal Q(\pi_{\mathrm{dia}})$.}  By the Schoenberg correspondence
\cite[Theorem~3.2.2]{berg1984harmonic}, the assumption that
\(e^{-\alpha\psi_i(|u-v|)}\) is positive definite for every \(\alpha>0\)
implies that $k_i(u,v):=\psi_i(|u-v|)$ 
is conditionally negative definite.  Since \(k_i(u,u)=0\), the Schoenberg
embedding theorem implies that \(d_i(u,v):=k_i(u,v)^{1/2}\) is a Hilbertian
pseudometric; equivalently, one may use the triangle inequality for
\(d_i\); see \cite[Theorem~1]{schoenberg1938metric} and
\cite[Proposition~3.3.2]{berg1984harmonic}.  Therefore,
\begin{equation}
\label{eq:negative-type-growth}
  \psi_i(|u-v|)
  =
  d_i(u,v)^2
  \le
  \bigl(d_i(u,0)+d_i(0,v)\bigr)^2
  \le
  2\psi_i(|u|)+2\psi_i(|v|).
\end{equation}

Let \((X,Y)\sim\pi_{\mathrm{dia}}\).  Under the diamond transport, the radii
\(|X|\) and \(|Y|\) are antimonotonically coupled.  Since the functions
\(\psi_i\) are nondecreasing, \(\psi_1(|X|)\) and \(\psi_2(|Y|)\) are also
antimonotonically coupled.  Applying the one-dimensional rearrangement theorem \cite[Theorem~2.9]{santambrogio2015optimal}, we have
\begin{align}
    \mathbb E_{\pi_{\mathrm{dia}}}
  [\psi_1(|X|)\psi_2(|Y|)]
  \le
  \mathbb E_\mu[\psi_1(|X|)]\mathbb E_\nu[\psi_2(|Y|)]<\infty .\label{t1}
\end{align}
Now let \((X',Y')\) be an independent copy of \((X,Y)\).  By
\eqref{eq:negative-type-growth},
\begin{align}
&\psi_1(|X-X'|)\psi_2(|Y-Y'|)  \le
4\bigl(\psi_1(|X|)+\psi_1(|X'|)\bigr)
 \bigl(\psi_2(|Y|)+\psi_2(|Y'|)\bigr).\label{t2}
\end{align}
Combining \eqref{t1} and \eqref{t2} yields \(\mathcal Q(\pi_{\mathrm{dia}})<\infty\).

\smallskip\noindent\textbf{Step V: uniqueness.} 
\sloppy It remains to prove uniqueness under the additional cosine-representation
assumption \eqref{eq:cosine-representation}.  In what follows, \(\mathcal Q\)
denotes the quadratic-form energy with cost
\(\psi_1(|x-x'|)\psi_2(|y-y'|)\). If \(\pi_0,\pi_1\) are two minimizers and \(\sigma=\pi_1-\pi_0\), then
\(\sigma\) has zero marginals. Denote its Fourier transform by 
\(\widehat\sigma(t,s)=\int e^{i(tx+sy)}\,\dd\sigma(x,y)\).
For \(n\ge1\), put
\(\Lambda_i^{(n)}:=\mathbf 1_{\{1/n\le |t|\le n\}}\Lambda_i\) and
\[
  \psi_i^{(n)}(|u|)
  =
  \int_{\R}(1-\cos(tu))\,\Lambda_i^{(n)}(\dd t).
\]
The measures \(\Lambda_i^{(n)}\) are finite, so Fubini gives
\[
\begin{aligned}
&\iint \psi_1^{(n)}(|x-x'|)\psi_2^{(n)}(|y-y'|)
\,\dd\sigma(x,y)\,\dd\sigma(x',y')\\
&\quad =
\iint_{\R^2}
\iint
(1-\cos(t(x-x')))(1-\cos(s(y-y')))
\,\dd\sigma(x,y)\,\dd\sigma(x',y')\,
\Lambda_1^{(n)}(\dd t)\Lambda_2^{(n)}(\dd s)\\
&\quad =
\iint_{\R^2}
\iint
\cos(t(x-x'))\cos(s(y-y'))
\,\dd\sigma(x,y)\,\dd\sigma(x',y')\,
\Lambda_1^{(n)}(\dd t)\Lambda_2^{(n)}(\dd s)\\
&\quad =
\frac12\iint_{\R^2}
\left(
|\widehat\sigma(t,s)|^2+|\widehat\sigma(t,-s)|^2
\right)
\,\Lambda_1^{(n)}(\dd t)\Lambda_2^{(n)}(\dd s)\\
&\quad =
\iint_{\R^2}
|\widehat\sigma(t,s)|^2
\,\Lambda_1^{(n)}(\dd t)\Lambda_2^{(n)}(\dd s).
\end{aligned}
\]
Here the second equality uses the zero one-dimensional marginals of
\(\sigma\), and the last equality uses the symmetry of
\(\Lambda_2^{(n)}\) under \(s\mapsto -s\).
Since
\(0\le \psi_i^{(n)}\uparrow \psi_i\) and by our moment assumptions, Lemma \ref{lemma:abs integrability}, and \eqref{eq:negative-type-growth},
dominated convergence on the left and monotone convergence on the right give
\begin{equation}
\label{eq:product-fourier-identity}
\begin{aligned}
&\iint \psi_1(|x-x'|)\psi_2(|y-y'|)
\,\dd\sigma(x,y)\,\dd\sigma(x',y')\\
&\qquad =
\iint_{\R^2} |\widehat\sigma(t,s)|^2
\,\Lambda_1(\dd t)\Lambda_2(\dd s)\ge0.
\end{aligned}
\end{equation}
If the last integral in \eqref{eq:product-fourier-identity} is zero, then
\(\widehat\sigma=0\) on a set dense in \(\R^2\), since
\(\Lambda_1\otimes\Lambda_2\) has full support.  Continuity of
\(\widehat\sigma\) and uniqueness of Fourier transforms give \(\sigma=0\).

 Recall that \(\sigma=\pi_1-\pi_0\).  Expanding the quadratic form at
\((\pi_0+\pi_1)/2\) gives
\begin{equation}
\label{eq:product-midpoint}
  \mathcal Q\!\left(\frac{\pi_0+\pi_1}{2}\right)
  =
  \frac12\mathcal Q(\pi_0)+\frac12\mathcal Q(\pi_1)
  -\frac14
  \iint \psi_1(|x-x'|)\psi_2(|y-y'|)
  \,\dd\sigma(x,y)\,\dd\sigma(x',y').
\end{equation}
The midpoint identity \eqref{eq:product-midpoint} and the nonnegativity from
\eqref{eq:product-fourier-identity} show that equality with the minimum
forces the quadratic term
\[
  \iint \psi_1(|x-x'|)\psi_2(|y-y'|)
  \,\dd\sigma(x,y)\,\dd\sigma(x',y')
\]
to vanish.  This gives \(\sigma=0\), proving uniqueness.
\end{proof}

\begin{proof}[Proof of Example~\ref{ex:product-profiles}]
For (i),  this follows directly from Example \ref{ex:pre}.

For (ii), \(s\mapsto(1+\lambda_i s)^{\theta_i}-1\) is a Bernstein function and
\(u\mapsto |u|^{a_i}\) is continuous negative definite.  The
Bernstein-function composition theorem
\cite[Theorem~3.2.9]{berg1984harmonic} gives continuous negative definiteness
of \(u\mapsto\psi_i(|u|)\).  Moreover,
\[
  (1+\lambda_i s)^{\theta_i}-1
  =
  c_{\theta_i}\int_0^\infty e^{-t}(1-e^{-\lambda_i ts})
  t^{-1-\theta_i}\,\dd t,
  \qquad c_{\theta_i}>0,
\]
and the stable cosine representation of \(1-e^{-\lambda_i t|u|^{a_i}}\) (see \eqref{eq:stable} below)
gives a full-support cosine representation for \(\psi_i\).  Finally,
\[
  \psi_i''(r)
  =
  \theta_i\lambda_i a_i r^{a_i-2}(1+\lambda_i r^{a_i})^{\theta_i-2}
  \bigl((a_i-1)+(a_i\theta_i-1)\lambda_i r^{a_i}\bigr)>0 .
\]
Thus Theorem~\ref{thm:taylor-products} applies.

For (iii), the stable characteristic function gives
\begin{align}
    1-e^{-\lambda |u|^r}
  =
  \int_{\R}(1-\cos(ut))g_{\lambda,r}(t)\,\dd t,\quad r\in(0,2],\label{eq:stable}
\end{align}
where \(g_{\lambda,r}\) is the strictly positive density of a symmetric stable
law \cite[Section~14]{sato1999levy}.  Hence the cosine representation has full
support.  Also
\[
  \psi_1''(r)
  =
  \lambda p r^{p-2}e^{-\lambda r^p}(p-1-\lambda p r^p)>0
\]
on \((0,D_x]\) under \(\lambda pD_x^p<p-1\), and similarly for \(\psi_2\).\footnote{Note that here (and for (iv) below), $\psi_i$ is not convex on $(0,\infty)$, but we may without loss of generality replace it by a convex function that is equal to $\psi_i$ on the compact supports of $\mu,\nu$.}
Thus Theorem~\ref{thm:taylor-products} applies. 

For (iv), the identity
\[
  \log(1+\lambda s)
  =
  \int_0^\infty (1-e^{-ts})e^{-t/\lambda}\,\frac{\dd t}{t}
\]
combined with the stable representation \eqref{eq:stable} gives continuous negative
definiteness and a full-support cosine representation.  Finally,
\[
  \psi_1''(r)
  =
  \frac{\lambda p r^{p-2}(p-1-\lambda r^p)}
       {(1+\lambda r^p)^2}>0
\]
on \((0,D_x]\) under \(\lambda D_x^p<p-1\), and similarly for \(\psi_2\).
Again Theorem~\ref{thm:taylor-products} applies.
\end{proof}

\subsection{Proof of results in Section \ref{33}}

Throughout this subsection, write
\[
  \mathcal Q_{p,q}(\pi):=
  \iint |x-x'|^p|y-y'|^q\,\dd\pi(x,y)\,\dd\pi(x',y')
\]
for the quadratic-form energy associated with the rectangular-power cost.

\begin{lemma}
\label{lem:finite-mixed}
In the same setup as Theorem~\ref{thm:q-rectangular}, if
\(\mathcal Q_{p,q}(\pi)<\infty\), then
\(\int |x|^p|y|^q\,\dd\pi(x,y)<\infty\).
\end{lemma}

\begin{proof}
Choose \(R>0\) such that \(m:=\pi([-R,R]^2)>0\).  Let
\(B=[-R,R]^2\) and \(A=\{|x|>2R,\ |y|>2R\}\).  If
\((x,y)\in A\) and \((x',y')\in B\), then
\(|x-x'|^p|y-y'|^q\ge 2^{-p-q}|x|^p|y|^q\).
It follows that
\[
  \mathcal Q_{p,q}(\pi)
  \ge
  m\,2^{-p-q}\int_A |x|^p|y|^q\,\dd\pi(x,y),
\]
so the mixed moment is finite on \(A\).  On \(A^c\), at least one of
\(|x|\) and \(|y|\) is at most \(2R\), and hence
\(|x|^p|y|^q\le(2R)^q|x|^p+(2R)^p|y|^q\).
The latter is \(\pi\)-integrable because the marginals are in
\(\mathcal P_p(\R)\) and \(\mathcal P_q(\R)\).
\end{proof}

\begin{lemma}
\label{lem:strict-negative-type}
Let \(0<a<2\).  Let \(\tau\) be a finite signed Borel measure on \(\R\) with
\(\tau(\R)=0\) and \(\int |x|^a\,\dd|\tau|(x)<\infty\).  Then
\begin{align}
     \iint |x-x'|^a\,\dd\tau(x)\,\dd\tau(x')\le0,\label{g}
\end{align}
and equality holds if and only if \(\tau=0\).
\end{lemma}

\begin{proof}
By \cite[Lemma~6]{kroll2022asymptotic}, applied with \(H=\R\) and
\(\beta=a\), the semimetric space \((\R,|\cdot-\cdot|^a)\) is of strong
negative type.  Since \(\tau(\R)=0\), the signed-measure form of this property
gives \eqref{g}. 
Moreover, equality in strong negative type can occur only when the positive
and negative parts of \(\tau\) agree.  Hence equality holds only if
\(\tau^+=\tau^-\), i.e. \(\tau=0\).  The converse is immediate.
\end{proof}

\begin{lemma}
\label{lem:q-rect-strict}
Let \(0<p,q<2\).  Let \(\sigma\) be a finite signed Borel measure on
\(\R^2\) whose two one-dimensional marginals are zero.  Assume that
\begin{align}
    \int_{\R^2} (1+|x|^p)(1+|y|^q)\,\dd|\sigma|(x,y)<\infty .\label{eq:mom}
\end{align}
Then
\[
  \mathcal E_{p,q}(\sigma):=
  \iint |x-x'|^p|y-y'|^q
  \,\dd\sigma(x,y)\,\dd\sigma(x',y')
\]
is absolutely finite, \(\mathcal E_{p,q}(\sigma)\ge0\), and equality holds if
and only if \(\sigma=0\).
\end{lemma}

\begin{proof}
Let \(d_p(x,x')=|x-x'|^p\) and \(d_q(y,y')=|y-y'|^q\).  Since
\(0<p,q<2\), these are semimetrics of strong negative type on \(\R\); see
\cite[Lemma~6]{kroll2022asymptotic}.  By Schoenberg's Hilbert-space
characterization of negative type, equivalently
\cite[Remark~2.7]{janson2021distance}, and the tensor-barycenter lemma
\cite[Lemma~3.8]{lyons2013distance}, as corrected in \cite[item~(ix) of the erratum]{lyons2018errata}, there are Hilbert spaces \(H,H'\) and
maps \(\varphi:\R\to H\), \(\psi:\R\to H'\) such that
\[
  d_p(x,x')=\|\varphi(x)-\varphi(x')\|_H^2,
  \qquad
  d_q(y,y')=\|\psi(y)-\psi(y')\|_{H'}^2 
\]
and that the tensor barycenter map
$\theta\mapsto \int \varphi(x)\otimes\psi(y)\,\dd\theta(x,y)$
is injective on finite signed measures whose total-variation marginals have
finite first moments in \(d_p\) and \(d_q\).
The moment assumption \eqref{eq:mom} makes
\(\int \varphi(x)\otimes\psi(y)\,\dd\sigma(x,y)\) well-defined.  Since
\(\sigma\) has total mass zero and zero one-dimensional marginals, expanding
the two squared Hilbert distances gives
\[
\begin{aligned}
\mathcal E_{p,q}(\sigma)
&=
\iint d_p(x,x')d_q(y,y')\,
\dd\sigma(x,y)\,\dd\sigma(x',y')  \\
&=
\iint
\|\varphi(x)-\varphi(x')\|_H^2
\|\psi(y)-\psi(y')\|_{H'}^2\,
\dd\sigma(x,y)\,\dd\sigma(x',y')  \\
&=
\iint
\bigl(\|\varphi(x)\|^2+\|\varphi(x')\|^2
      -2\langle\varphi(x),\varphi(x')\rangle\bigr) \\
&\qquad\qquad\times
\bigl(\|\psi(y)\|^2+\|\psi(y')\|^2
      -2\langle\psi(y),\psi(y')\rangle\bigr)\,
\dd\sigma(x,y)\,\dd\sigma(x',y')  \\
&=
4\iint
\langle\varphi(x),\varphi(x')\rangle
\langle\psi(y),\psi(y')\rangle\,
\dd\sigma(x,y)\,\dd\sigma(x',y')  \\
&=
4\left\|
  \int \varphi(x)\otimes\psi(y)\,\dd\sigma(x,y)
\right\|_{H\otimes H'}^2
\ge0 .
\end{aligned}
\]
In the fourth equality, every term except the product of the two inner
products vanishes because \(\sigma(\R^2)=0\) and the two one-dimensional
marginals of \(\sigma\) are zero.
If equality holds, then the injectivity gives \(\sigma=0\).
\end{proof}

\begin{proof}[Proof of Theorem~\ref{thm:q-rectangular}]
We split our proof into five steps.

\smallskip\noindent\textbf{Step I: minimization for \(1<p,q\le2\).}
This follows directly from  Theorem \ref{thm:taylor-products} applied with $\psi_1(r)=r^p$ and $\psi_2(r)=r^q$. The moment constraints from Theorem \ref{thm:taylor-products}(ii) follow from our assumptions.

\smallskip\noindent\textbf{Step II: uniqueness for \(1<p,q<2\), proving (i).}
Let \(\pi_0,\pi_1\) be minimizers.  They have finite \(\mathcal Q_{p,q}\)-energy, so
Lemma~\ref{lem:finite-mixed} gives
\(\int |x|^p|y|^q\,\dd\pi_i(x,y)<\infty\) for \(i=0,1\).
Set \(\sigma=\pi_1-\pi_0\).  The signed measure \(\sigma\) has zero
marginals.  The estimate
\[
  |x-x'|^p|y-y'|^q\le
  2^{p+q-2}(|x|^p+|x'|^p)(|y|^q+|y'|^q)
\]
expands into four terms, each finite with respect to
\((\pi_0+\pi_1)\otimes(\pi_0+\pi_1)\) by the mixed moments above and the
marginal \(p\)- and \(q\)-moment assumptions.  Since
\(|\sigma|\le\pi_0+\pi_1\), this gives the absolute integrability required in
Lemma~\ref{lem:q-rect-strict}.
Let \(\bar\pi=(\pi_0+\pi_1)/2\).  Expanding the quadratic form gives
\begin{equation}
\label{eq:qrect-midpoint}
\begin{aligned}
\mathcal Q_{p,q}(\bar\pi)
&=
\frac12\mathcal Q_{p,q}(\pi_0)+\frac12\mathcal Q_{p,q}(\pi_1)
-\frac14
\iint |x-x'|^p|y-y'|^q\,\dd\sigma(x,y)\,\dd\sigma(x',y') \\
&=
\inf_{\pi\in\Pi(\mu,\nu)}\mathcal Q_{p,q}(\pi)
-\frac14\mathcal E_{p,q}(\sigma).
\end{aligned}
\end{equation}
Since \(\bar\pi\) is feasible, \eqref{eq:qrect-midpoint} implies that
\(\mathcal E_{p,q}(\sigma)\le0\).  Lemma~\ref{lem:q-rect-strict} gives
\(\mathcal E_{p,q}(\sigma)\ge0\), and hence
\(\mathcal E_{p,q}(\sigma)=0\).  The strict equality statement in the lemma
gives \(\sigma=0\), so the minimizer is unique.  By Step I, it is the diamond
transport.

\smallskip\noindent\textbf{Step III: boundary characterization for \(1<p<2=q\), proving the minimizer description in (ii).}
Consider \(1<p<2=q\), and let \(\pi=\law(X,Y)\in\Pi(\mu,\nu)\).  Since the
diamond value
\[
  \mathcal Q_{p,2}(\pi_{\mathrm{dia}})
  =
  \iint |x-x'|^p|y-y'|^2
  \,\dd\pi_{\mathrm{dia}}(x,y)\,\dd\pi_{\mathrm{dia}}(x',y')
\]
is finite by Step I, only finite-energy couplings can be
minimizers.  For such \(\pi\), define
\[
  \Phi_p(x):=\int_\R |x-u|^p\,\dd\mu(u),
  \qquad
  \tau(A):=\int_{A\times\R}y\,\dd\pi(x,y).
\]
Then \(\tau\) is a finite signed measure and \(\tau(\R)=0\).  Expanding
\(|y-y'|^2\) gives
\begin{equation}
\label{eq:p2-expansion}
  \mathcal Q_{p,2}(\pi)
  =
  2\int \Phi_p(x)y^2\,\dd\pi(x,y)
  -2\iint |x-x'|^p\,\dd\tau(x)\,\dd\tau(x').
\end{equation}
The formula \eqref{eq:p2-expansion} separates the two constraints: 

\begin{itemize}
    \item First, since \(\mu\) is
symmetric, \(\Phi_p(x)=\phi_p(|x|)\), where \(\phi_p\) is strictly increasing
on \([0,\infty)\).\footnote{Indeed, for fixed \(a\ge0\), the function
\((r+a)^p+|r-a|^p\) is strictly increasing in \(r\ge0\), and integration
against the law of \(|X|\) preserves this monotonicity.}
By the strict supermodular rearrangement theorem
\cite[Theorem~3.1]{puccetti2015extremal}, applied to the laws of
\(|X|\) and \(|Y|\), 
\[
  \int \Phi_p(x)y^2\,\dd\pi(x,y)
  \ge
  \int \phi_p(r)t^2\,\dd\gamma^-(r,t),
\]
with equality exactly when \(\law(|X|,|Y|)=\gamma^-\). 
\item  Second,
Lemma~\ref{lem:finite-mixed} and $|y|\leq 1+y^2$ together give
\(\int |x|^p\,\dd|\tau|(x)<\infty\), so
Lemma~\ref{lem:strict-negative-type} gives
\[
  \iint |x-x'|^p\,\dd\tau(x)\,\dd\tau(x')\le0,
\]
with equality only when \(\tau=0\).
The condition \(\tau=0\) is equivalent to \(\E_\pi[Y\mid X]=0\) \(\pi\)-a.s.
\end{itemize}

The diamond transport has absolute-value law \(\gamma^-\) and satisfies
\(\tau=0\), so \eqref{eq:p2-expansion} proves the characterization in (ii).

\smallskip\noindent\textbf{Step IV: boundary uniqueness for (ii).}
Assume \(1<p<2=q\).  If \(\gamma^-\) is Monge from \(\law(|X|)\) to
\(\law(|Y|)\) away from zero, let \(S:(0,\infty)\to[0,\infty)\) be a
corresponding Borel map.  For any minimizer, Step III gives
\(\law(|X|,|Y|)=\gamma^-\) and \(\E_\pi[Y\mid X]=0\).  Thus,
\(|Y|=S(|X|)\) whenever \(|X|>0\).  The symmetry of \(\mu\) makes the sign of
\(X\) conditionally uniform given \(|X|>0\).  If \(S(|X|)>0\), then
\[
  \E_\pi[\operatorname{sgn}(Y)\mid X]
  =S(|X|)^{-1}\E_\pi[Y\mid X]=0,
\]
so the sign of \(Y\) is conditionally uniform given \(X\); if \(S(|X|)=0\),
then \(Y=0\).  Thus, the part of the minimizer on \(\{|X|>0\}\) agrees with
the diamond transport.  On \(\{X=0\}\), the absolute-value law fixes the
conditional law of \(|Y|\).  The already identified part has a symmetric
\(Y\)-marginal, and the full \(Y\)-marginal is symmetric, so the remaining
projection to the \(y\)-coordinate is also symmetric.  Therefore, the
remaining part also agrees with the diamond transport, and uniqueness holds.

Conversely, suppose \(\gamma^-\) is not Monge from \(\law(|X|)\) to
\(\law(|Y|)\) away from zero.  Disintegrate
\(\gamma^-(\dd\rho,\dd\theta)=\alpha(\dd\rho)\kappa_\rho(\dd\theta)\) with
respect to its first marginal.  Failure of the Monge property means that
\(\kappa_\rho\) is not a Dirac mass for a set of \(\rho>0\) with positive $\alpha$-measure.  Hence,
there is a bounded Borel \(f(\rho,\theta)\), not \(\gamma^-\)-a.e.~zero and
supported on \(\rho>0\), such that
\[
  \int f(\rho,\theta)\,\kappa_\rho(\dd\theta)=0
  \quad\text{for }\alpha\text{-a.e. }\rho .
\]
After multiplying \(f\) by a small scalar, assume \(|f|\le1\).  Sample
\((\rho,\theta)\sim\gamma^-\), choose \(\eta\) as an independent Rademacher
sign, and choose \(\varepsilon\) conditionally by
\[
  \mathbb P(\varepsilon=1\mid \rho,\theta)=\frac{1+f(\rho,\theta)}2,
  \qquad
  \mathbb P(\varepsilon=-1\mid \rho,\theta)=\frac{1-f(\rho,\theta)}2 .
\]
Then \(X=\varepsilon\rho\) has law \(\mu\), \(Y=\eta\theta\) has law \(\nu\),
and \(\law(|X|,|Y|)=\gamma^-\).  Indeed,
\(\E[\varepsilon\mid\rho]=0\) by the conditional mean-zero property of
\(f\), while \(\eta\) is an independent balanced sign.  If
\(\tilde\pi:=\law(X,Y)\), then \(\E_{\tilde\pi}[Y\mid X]=0\) because
\(\eta\) has conditional mean zero even after conditioning on \(X\), but
\(\tilde\pi\) is not the diamond law
because the sign of \(X\) is not conditionally independent of \(\theta\)
given \(\rho\).  Hence, uniqueness fails.  This proves the
uniqueness criterion in (ii).

\smallskip\noindent\textbf{Step V: the boundary \(p=q=2\), proving (iii).}
Let \(\pi=\law(X,Y)\) and $\E=\E_\pi$ in the following.  For finite-energy \(\pi\), expansion gives
\begin{equation}
\label{eq:q2-expansion}
  \mathcal Q_{2,2}(\pi)
  =
  2\E[X^2Y^2]+2\E[X^2]\E[Y^2]+4(\E[XY])^2.
\end{equation}
By \cite[Theorem~3.1]{puccetti2015extremal},
\(\E[X^2Y^2]\ge\int r^2t^2\,\dd\gamma^-(r,t)\), with equality exactly when
\(\law(|X|,|Y|)=\gamma^-\).  Therefore, minimizers are exactly the couplings
with \(\law(|X|,|Y|)=\gamma^-\) and \(\E[XY]=0\).

For the uniqueness statement, fix a minimizer and write
\((\rho,\theta)=(|X|,|Y|)\), so \((\rho,\theta)\sim\gamma^-\).  On the set
\(\rho\theta>0\), write \(X=\varepsilon\rho\) and \(Y=\eta\theta\), where
\(\varepsilon,\eta\in\{-1,1\}\).  The fixed symmetric marginals impose the
one-dimensional sign-balance constraints, while the condition \(\E[XY]=0\)
is
\[
  \int \rho\theta\,\E[\varepsilon\eta\mid\rho,\theta]\,
  \dd\gamma^-(\rho,\theta)=0.
\]

Suppose first that \(\gamma^-\) is Monge from \(\law(|X|)\) to
\(\law(|Y|)\) away from zero, the transpose of \(\gamma^-\) is Monge from
\(\law(|Y|)\) to \(\law(|X|)\) away from zero, and that
\(\gamma^-|_{(0,\infty)^2}\) is either zero or
concentrated at a single point.  On the axes, one coordinate is zero.  Once
the positive-absolute-value part is fixed, the symmetry of the corresponding
one-dimensional marginal forces the remaining one-coordinate sign
distribution on the axes to be symmetric, exactly as in the diamond
transport.  If the positive-absolute-value part is concentrated at
\((r_0,t_0)\), the two Monge properties imply that no other mass has
\(|X|=r_0\) or \(|Y|=t_0\).  The two marginal sign constraints make the signs
at \((r_0,t_0)\) uniform individually, and \(\E[XY]=0\) forces their
correlation to vanish.  Thus, the four sign pairs at \((r_0,t_0)\) have
probability \(1/4\), and the only minimizer is \(\pi_{\mathrm{dia}}\).

Conversely, if \(\gamma^-\) is not Monge from \(\law(|X|)\) to
\(\law(|Y|)\) away from zero, the sign-bias construction in Step IV gives a
non-diamond minimizer.  The same holds if the transpose of \(\gamma^-\) is
not Monge from \(\law(|Y|)\) to \(\law(|X|)\) away from zero.  It remains to
consider the case where both Monge properties hold but
\(\gamma^-|_{(0,\infty)^2}\) is neither zero nor
concentrated at a single point.  Then one can choose a bounded nonzero Borel
function \(k\), supported on \((0,\infty)^2\), such that
\(\int rt\,k(r,t)\,\dd\gamma^-(r,t)=0\) and \(|k|\le1\).\footnote{For
instance, take two disjoint positive-mass sets in \((0,\infty)^2\) and
subtract constants so that the \(rt\,\gamma^-\)-mean is zero.}  Given
\((\rho,\theta)\sim\gamma^-\), choose signs by
\[
  \mathbb P(\varepsilon=a,\eta=b\mid \rho,\theta)
  =
  \frac{1+ab\,k(\rho,\theta)}4,\qquad a,b=\pm1.
\]
This keeps both one-dimensional marginals symmetric, keeps the law of the
absolute values equal to \(\gamma^-\), and gives \(\E[XY]=0\), but it is not
the diamond coupling.  Therefore, uniqueness fails in all remaining cases.
\end{proof}

\begin{proof}[Proof of Example~\ref{ex:q2-uniform}]
For any \(\pi\in\Pi(\U[-1,1],\U[-1,1])\), with \((X,Y)\sim\pi\), the
expansion \eqref{eq:q2-expansion} gives
\[
  \mathcal Q_{2,2}(\pi)
  =
  2\E[X^2Y^2]+2\E[X^2]\E[Y^2]+4(\E[XY])^2.
\]
Here, \(\E[X^2]=\E[Y^2]=1/3\).  Also, \(X^2\) and \(Y^2\) have quantile
function \(t\mapsto t^2\), so the rearrangement inequality gives
\(\E[X^2Y^2]\ge\int_0^1t^2(1-t)^2\,\dd t=1/30\).  Hence, every coupling has
energy at least \(13/45\).

For the construction in \eqref{eq:q2-sign-construction}, the conditional
sign marginals are uniform.
The variables \(X,Y\) are therefore distributed as
\(\U[-1,1]\), and
\((|X|,|Y|)=(R,1-R)\) has the \(\piant\) absolute-value law.  Also,
\(\E[XY]=\int_0^1 r(1-r)(r-1/2)\,\dd r=0\).
Thus, its energy is \(13/45\), so it is a minimizer.
However, it is not the diamond coupling.  For the centered diamond coupling,
conditional on \(|X|=r\), the four sign pairs
\((\operatorname{sgn}X,\operatorname{sgn}Y)\) are equiprobable.  In the
coupling above,
\(\mathbb P(XY>0\mid |X|=r)=(1+r-1/2)/2\),
which is not identically \(1/2\).
\end{proof}

\section{Conclusion}\label{sec:conclusion}
The results in this paper identify new type-XX QOT problems which have an explicit diamond minimizer.  They
extend the diamond theory of \cite{wangzhangQOT} by allowing different
one-dimensional profiles, weaker moment constraints, and further examples. The proofs separate two mechanisms that are
merged in the original rectangular example: convexity of the QOT quadratic
form and
the quadrant rearrangement structure of the diamond transport. Using negative-type properties, we also provide sufficient conditions for the uniqueness of these QOT solutions and fully characterize the uniqueness for the rectangular-type cost $|x-x'|^p|y-y'|^q,~p,q\in(1,2]$. 
The DOT counterpart in Section~\ref{sec:dot} shows that the same diamond geometry also appears outside QOT, both as a class of minimizers and, after calibration, as a unique minimizer.

\section*{Acknowledgments}

Ruodu Wang was supported by the Natural Sciences and Engineering Research Council of Canada (CRC-2022-00141, RGPIN-2024-03728).
Zhenyuan Zhang was supported by the Jump Trading Fellowship.

\bibliographystyle{plain}
\bibliography{qot_combined_open_problem_solutions}

\end{document}